\documentclass[reqno]{amsart}
\usepackage[backend=bibtex, sorting=none, bibstyle=alphabetic, citestyle=alphabetic, sorting=nyt, maxbibnames=99, giveninits=true, isbn=false, url=false, doi=false, eprint=false]{biblatex}  
\renewbibmacro{in:}{}
\bibliography{references.bib}
\usepackage{amssymb, calrsfs, graphics, graphicx, enumerate, enumitem, url, xcolor, hyperref, mathrsfs, listings, comment}
\usepackage{tabularx}
\usepackage[ruled, lined, linesnumbered, longend]{algorithm2e}
\usepackage[justification=centering]{caption}
\newtheorem{theorem}{Theorem}[section]
\newtheorem{lemma}[theorem]{Lemma}

\newtheorem{corollary}[theorem]{Corollary}
\numberwithin{equation}{section}

\theoremstyle{remark}
\newtheorem*{remark}{Remark}

\lstdefinestyle{CStyle}{
    basicstyle=\footnotesize,
    breakatwhitespace=false,         
    breaklines=true,                 
    captionpos=b,                    
    keepspaces=true,                 
    numbers=left,                    
    numbersep=5pt,                  
    showspaces=false,                
    showstringspaces=false,
    showtabs=false,                  
    tabsize=2,
    language=C
}

\title[Explicit bounds for $\zeta$ and a new zero free region]{Explicit bounds for the Riemann zeta function \\and a new zero free region}
\author[C. Bellotti]{Chiara  Bellotti}
\address{School of Science\\
The University of New South Wales, Canberra, Australia}
\email{c.bellotti@adfa.edu.au}
\date\today
\keywords{Riemann zeta function, explicit bounds, zero-free region, Vinogradov integral}
\subjclass[2020]{Primary 11M06, 11N05, 11L15; Secondary 11D72,
11M35.}
\begin{document}

\begin{abstract}
We prove that $|\zeta(\sigma+it)|\le 70.7 |t|^{4.438
    (1-\sigma)^{3/2}}\log^{2/3}|t|$ for $1/2\le\sigma\le 1$ and $|t|\ge 3$. As a consequence, we improve the explicit zero-free region for $\zeta(s)$, showing that $\zeta(\sigma+it)$ has no zeros in the region $\sigma \geq 1-1 /\left(54.004(\log |t|)^{2 / 3}(\log \log |t|)^{1 / 3}\right)$ for $|t| \geq 3$ and asymptotically in the region $\sigma \geq 1-1 /\left(48.0718(\log |t|)^{2 / 3}(\log \log |t|)^{1 / 3}\right)$ for $|t|$ sufficiently large.
\end{abstract}
\maketitle
\section{Introduction}
Let $\zeta(s)$ be the Riemann zeta function, with $s=\sigma+it$ a complex variable. It is known that all the non trivial zeros of $\zeta(s)$ have real part $\sigma\in (0,1)$. Detecting zero free regions for $\zeta(s)$ inside the critical strip $0<\sigma<1$ is an open problem that has always caught much interest in analytic number theory. Great effort has been put in trying to find both asymptotically and explicit regions inside the critical strip where there are no zeros of $\zeta(s)$. The classical zero-free region is of the form $\sigma>1-1/(R_0\log |t|)$, where $R_0$ is a positive constant. The best known result of this form is due to Mossinghoff, Trudgian and Yang \cite{mossinghoff_explicit_2022} with $R_0=5.558691$ for every $|t|\ge 2$ (see \cite{stechkin_zeros_1970, rosser_sharper_1975,kondratev_some_1977,kadiri_region_2005,jang_note_2014,mossinghoff_nonnegative_2014} for previous results). Littlewood zero-free region \cite{littlewood_researches_1922} is instead of the form of $\sigma>1-\log\log |t|/C_1\log|t|$, where $C_1$ is a positive constant. It has been made first explicit by Yang \cite{yang2023explicit}, who found $C_1=21.432$ for $|t|\ge 3$.\\Asymptotically larger zero-free regions for the Riemann zeta function $\zeta(s)$, known as Korobov--Vinogradov zero-free regions,  are of the form \begin{equation}\label{kvzerofree}
    \sigma>1-\frac{1}{C_2(\log |t|)^{2/3}(\log\log|t|)^{1/3}},
\end{equation}
where $C_2$ is a positive constant and are due to the method of Korobov \cite{korobov_abschatzungen_1958} and Vinogradov \cite{vinogradov_eine_1958}, in which the main tool is an upper bound on $|\zeta(\sigma+it)|$ when $\sigma$ is near to the line $\sigma=1$. Upper bounds of this form were first made explicit by Richert \cite{richert_zur_1967}, who used Korobov-Vinogradov method to prove that \begin{equation}\label{upperbound}
|\zeta(\sigma+i t)| \leq A|t|^{B(1-\sigma)^{3 / 2}} \log ^{2 / 3}|t| \qquad|t| \geq 2,\  \frac{1}{2} \leq \sigma \leq 1
\end{equation}
with $B=100$ and $A$ a certain absolute constant. Although smaller values for $B$ were already found (see \cite{kulas_refinement_1999}), the first completely explicit bound of the form \eqref{upperbound} is due to Cheng \cite{cheng99}, with $A=175$ and $B=46$. This estimate was further improved in 2002 by Ford \cite{ford_vinogradovs_2002} to $A=76.2$ and $B=4.45$ for every $|t|\ge 3$ and $1/2\le \sigma \le 1$.\\
Explicit bounds of the form \eqref{upperbound} play a fundamental role in detecting explicit Korobov--Vinogradov zero-free regions for $\zeta(s)$. There are several results regarding the value of the constant $C_2$ in \eqref{kvzerofree} and the current best known estimate is due to Mossinghoff, Trudgian and Yang \cite{mossinghoff_explicit_2022}, with $C_2=55.241$ for every $|t|\ge 3$ (see \cite{cheng_explicit_2000,ford_zero_2002,ford_zero_2022} for previous results).\\
In this paper, we will improve the values for both the constants $A,B$ in \eqref{upperbound} and, as a consequence, we find an improved Korobov-Vinogradov zero-free region for $\zeta(s)$. More precisely, denoting with $\zeta(s, u)$ the Hurwitz zeta function $\zeta(s, u)=\sum_{n=0}^{\infty}(n+u)^{-s}$ defined for every $\Re s>1$ and $0<u\le 1$, we will prove the following result for both $\zeta(s)=\zeta(s,1)$ and a generic Hurwitz function $\zeta(s,u)$.
\begin{theorem}\label{theorem2}
    The following estimate holds for every $|t|\ge 3$ and $\frac{1}{2}\le \sigma\le 1$:
\begin{equation}
    \begin{aligned}
    |\zeta(\sigma+it)|&\le A |t|^{B
    (1-\sigma)^{3/2}}\log^{2/3}|t|\\
        \left|\zeta(\sigma+i t, u)-u^{-s}\right| &\leq A |t|^{B(1-\sigma)^{3 / 2}} \log ^{2 / 3} |t|, \qquad 0<u \leq 1,
    \end{aligned}
\end{equation}with $A=70.6995$ and $B=4.43795$.
\end{theorem}
The bound on $\zeta(s,u)$ found in Theorem \ref{theorem2} might be useful to bound Dirichlet $L$-functions due to the relation $L(s, \chi)=q^{-s} \sum_{m=1}^q \chi(m) \zeta(s, m / q)$, where $\chi$ is a Dirichlet character modulo $q$.\\
Although the new values $A=70.6995$ and $B=4.43795$ found in Theorem \ref{theorem2} are modest improvements on those found by Ford in \cite{ford_vinogradovs_2002} ($A=76.2$ and $B=4.45$ respectively), their importance relies on the fact that just a small improvement for both $A$ and $B$ can lead to improvements in several other results in analytic number theory. In particular, they have many applications in finding improved estimates for Korobov--Vinogrado zero-free region and in estimating the error term in the prime number theorem, both in an effective and ineffective way.\\
Furthermore, improvements on $A$ and $B$ only of the size found in Theorem \ref{theorem2} were already expected. In \cite{ford_vinogradovs_2002} (Section 8, point $1$), Ford already predicted that an optimization of the argument involving the Vinogradov integral would have lower $B$ by less than $0.02$, which is consistent with our new value for $B=4.43795$ found in Theorem \ref{theorem2}.\\

As already mentioned, an immediate consequence of Theorem \ref{theorem2} is a new explicit zero free region for the Riemann zeta function.
\begin{theorem}\label{theorem6}
    There are no zeros of $\zeta(\sigma+i t)$ for $|t| \geq 3$ and
$$
\sigma \geq 1-\frac{1}{54.004(\log |t|)^{2 / 3}(\log \log |t|)^{1 / 3}}.
$$
\end{theorem}
Theorem \ref{theorem2} has also an influence on asymptotically Korobov-Vinogradov zero-free regions of the form\begin{equation}\label{askv}
\sigma \geq 1-\frac{1}{c(\log |t|)^{2 / 3}(\log \log |t|)^{1 / 3}},
\end{equation}where $c$ is a positive constant for $|t|$ sufficiently large. More precisely, we get $c=48.0718$, improving $c$ on the current best value of $48.1588$ due to Mossinghoff, Trudgian and Yang \cite{mossinghoff_explicit_2022} (see \cite{ford_zero_2002} for previous results). 
\begin{theorem}\label{theoremasimp}
    For sufficiently large $|t|,$ there are no zeros of $\zeta(\sigma+it)$ with
    \begin{equation}\label{asimpzf}
\sigma \geq 1-\frac{1}{48.0718(\log |t|)^{2 / 3}(\log \log |t|)^{1 / 3}}.
\end{equation}
\end{theorem}
As in \cite{mossinghoff_explicit_2022}, the proofs of Theorem \ref{theorem6} and Theorem \ref{theoremasimp} rely on some non-negative trigonometric polynomials. We recall that, for every $K\ge 2$, a $K$-th degree non-negative trigonometric polynomial $P_K(x)$ is defined as  \begin{equation}\label{deftrpol}
    P_K(x)=\sum_{k=0}^K b_k \cos (k x)
\end{equation}
where $b_k$ are constants such that $b_k \geq 0, b_1>b_0$ and $P_K(x) \geq 0$ for all real $x$.\\
Finally, one can use Theorem \ref{theoremasimp} to improve on the error term in the prime number theorem. As per Ford \cite{ford_vinogradovs_2002}, we get the following estimate for the error term in the prime number theorem
$$
\pi(x)-\operatorname{li}(x) \ll x \exp \left\{-d(\log x)^{3 / 5}(\log \log x)^{-1 / 5}\right\},
$$with $$
d=\left(\frac{5^6}{2^2 \cdot 3^4 \cdot c^3}\right)^{1 / 5},
$$
where $c$ is the constant in \eqref{askv}. Using the new value for $c$ found in Theorem \ref{theoremasimp} we obtain $d=0.212579$, which is a slightly improvement on $d=0.2123$ found by Mossinghoff, Trudgian and Yang \cite{mossinghoff_explicit_2022}.\\

As in \cite{ford_vinogradovs_2002}, Theorem \ref{theorem2} follows from a uniform upper bound on the sum \begin{equation}\label{Snt}
  S(N, t) := \max_{0 < u \le 1}\max_{N < R \le 2N}\left|\sum_{N \le n \le R}\frac{1}{(n + u)^{it}}\right|,
\end{equation}
where $N$ is a positive integer and $t\ge N$. This upper bound is of the form 
\[S(N,t)\le c N^{1-1/(u\lambda^2)},
\]with $u,c>0$ and $\lambda=\log t/\log N.$ The strength of Ford's argument in \cite{ford_vinogradovs_2002} relies on some explicit estimates for both the Vinogradov integral and a quantity that counts the number of solutions of incomplete Diophantine systems that, combined with estimates for the exponential sum $S(N,t)$, give a completely explicit uniform upper bound on $S(N,t)$. Explicit bounds for the Vinogradov integral were further improved by Preobrazhenski\u{\i} \cite{preobrazhenskii_new_2011} in 2011 and Steiner in 2019 \cite{steiner_new_2019} (see \cite{hua_improvement_1949, stechkin_zeros_1970, tyr87, arkhipov2004trigonometric} for previous results).\\ We recall that the Vinogradov integral is defined as\begin{equation}\label{vinogradovintegral}
J_{s, k}(P)=\int_{[0,1]^k}\left|\sum_{1 \leq x \leq P} e\left(\alpha_1 x+\cdots+\alpha_k x^k\right)\right|^{2 s} d \boldsymbol{\alpha}
\end{equation}where $\boldsymbol{\alpha}=\left(\alpha_1, \ldots, \alpha_k\right)$ and $e(z)=e^{2 \pi i z}$, or equivalently, $J_{s, k}(P)$ is defined as the number of solutions of the simultaneous equations \begin{equation*}
\sum_{i=1}^s\left(x_i^j-y_i^j\right)=0 \quad(1 \leq j \leq k) ; \quad 1 \leq x_i, y_i \leq P.
\end{equation*}Regarding incomplete systems, we denote with $J_{s, k, h}(\mathscr{B})$ the number of solutions of the system \begin{equation*}
\sum_{i=1}^s\left(x_i^j-y_i^j\right)=0 \quad(h \leq j \leq k) ; \quad x_i, y_i \in \mathscr{B},
\end{equation*}where $\mathscr{B}$ is a suitable set. For our purpose, we will use the definition given by Ford in \cite{ford_vinogradovs_2002} $\mathscr{B}=\mathscr{C}(P,R)$ where $\mathscr{C}(P,R)$ is the set of integers $\le P$ composed only of prime factors in $(\sqrt{R}, R]$. Hence, $J_{s, k, h}(\mathscr{C}(P, R))$ is defined as\begin{equation}\label{defincomplete}
J_{s, k, h}(\mathscr{C}(P, R))=\int_{[0,1]^t}|f(\boldsymbol{\alpha})|^{2 s} d \boldsymbol{\alpha},
\end{equation}where$$
f(\boldsymbol{\alpha})=f(\boldsymbol{\alpha} ; P, R)=\sum_{x \in \mathscr{C}(P, R)} e\left(\alpha_h x^h+\cdots+\alpha_k x^k\right)
$$and $\boldsymbol{\alpha}=\left(\alpha_h, \cdots, \alpha_k\right)$.\\
In his paper \cite{ford_vinogradovs_2002}, Ford estimated $S(N,t)$ working separately and with different techniques on three different ranges of $\lambda$. More precisely, he considered the ranges $\lambda\le 87$, $87\le \lambda\le 220$ and $\lambda\ge 220$, where the critical case happens to be around $\lambda=87$. However, it is possible to shift the critical case to $\lambda=84$, with a consequent improved upper bound on $S(N,t)$ and so improved values of $A$ and $B$. Indeed, in his paper \cite{ford_vinogradovs_2002} (p. 590), Ford stated a variant of his argument to find even better estimates for the Vinogradov integral when $k$ is small. These results, combined with Preobrazhenski\u{\i}'s argument for $k\ge 90000$ in \cite{preobrazhenskii_new_2011} and some explicit bounds for the Vinogradov integral when $s\ll k^2$ due to Tyrina \cite{tyr87}, will give the following explicit bounds for $J_{s,k}$.
\begin{theorem}\label{newboundsvino}
  Let $k$ and $s$ be integers with $k \geq 90000$ and \[
    0.138128k^2\le s\le \frac{1}{2}k^2\left(0.6494+\log\left(\frac{8k}{25D}\right)-\left(\frac{1-D/k}{2-D/k}\right)\frac{D}{k}+\frac{2.051}{k}\right),
    \]with $1\le D\le 0.4k$. Then
$$
J_{s, k}(P) \leq  k^{2.055 k^3-0.414 k^2+3 n k} 1.06^{n k^2+2\left(n^2-n\right) k+0.099912 k^3}P^{2 s-k(k+1) / 2+\Delta_s} \quad(P \geq 1),
$$
where, denoting with $(s \bmod k)$ the unique integer $u$ such that $0 \leq u<k$,
\begin{equation}
\begin{aligned}
\Delta_s \leq & \max\left(\frac{8}{25} k^2  \exp \left(0.6494-\frac{2(s-k)}{k^2}-\left(\frac{1-D / k}{2-D / k}\right) \frac{D}{k}+\frac{2.051}{ k}+\frac{2(s \bmod k)}{k^3}\right),\right.\\&\ \ \quad\ \  \left.Dk\exp\left(\frac{2(s \bmod k)}{k^2}\left(-1+\frac{1}{k}\right)\right)\right).
\end{aligned}
\end{equation}
    Furthermore, if $k \geq 129$, there is an integer $s \leq \rho k^2$ such that for $P \geq 1$,
$$
J_{s, k}(P) \leq k^{\theta k^3} P^{2 s-\frac{1}{2} k(k+1)+0.001 k^2},
$$
with
\[
\begin{cases}
\rho = 3.177207,\quad\theta = 2.40930,&\text{if }129 \le k \le 137\\ \\
\rho = 3.177527,\quad\theta = 2.39529,&\text{if }138 \le k \le 139\\ \\
\rho = 3.178551,\quad\theta = 2.39167,&\text{if }140\le  k \le 146\\ \\
\rho = 3.178871,\quad\theta = 2.38259,&\text{if } 147 \le k \le 148\\ \\
\rho = 3.181869,\quad\theta = 2.37929,&\text{if }149 \le k \le 170\\ \\
\rho = 3.184127,\quad \theta = 2.35334,&\text{if }171 \le k \le 190\\ \\
\rho = 3.192950,\quad \theta = 2.33313,&\text{if } 191\le k\le 339\\ \\
\rho = 3.196497,\quad \theta = 2.24352,&\text{if } 340\le k\le 499\\ \\
\rho =  3.205502,\quad \theta = 1.77775,&\text{if } 500\le k<90000\\ \\
\rho = 3.208630,\quad \theta = 2.17720,&\text{if }  k\ge 90000.
\end{cases}
\]
\end{theorem}Using these new bounds for the Vinogradov integral we will prove the following result.\begin{theorem}\label{theorem1}
Suppose $N\ge 2$ is a positive integer, $N \leq t$ and set $\lambda=\frac{\log t}{\log N}$. Then
$$
S(N, t) \leq 8.7979 N^{1-1 /\left(132.94357\lambda^2\right)}.
$$
\end{theorem}
In proving both Theorem \ref{newboundsvino} and Theorem \ref{theorem1}, we tried to make near-optimal choices for all parameters that are used in the argument. Hence, it seems unlikely that a substantial improvement in the $B$ constant in Theorem \ref{theorem2} can be obtained via better choices of parameters alone.\\

An immediate corollary of Theorem \ref{theorem1} involving Dirichlet characters is the following one.
\begin{corollary}
Suppose $\chi$ is a Dirichlet character modulo $q$, where $q \leq N$ and $2 \leq N \leq q t$. Then
$$
\max _{N<R \leq 2 N}\left|\sum_{N<n \leq R} \chi(n) n^{-i t}\right| \leq 9.7979 \frac{\phi(q)}{q} N e^{-\frac{\log ^3(N / q)}{132.94357 \log ^2 t}}.
$$
\end{corollary}The proof is the same as that of Corollary 2A in \cite{ford_vinogradovs_2002}.

\section{Background}
We recall some preliminary results we will use later in the proof of the theorems.\\
 First of all, given a fixed $k$, suppose $0 \leq d \leq k-1$ and $T$ is a positive integer. The $k$-tuple of polynomials $\boldsymbol{\Psi}=\left(\Psi_1, \ldots, \Psi_k\right) \in \mathbb{Z}[x]^k$ is said to be of type $(d, T)$ if $\Psi_j$ is identically zero for $j \leq d$, and for some integer $m \geq 0$, when $j>d,\  \Psi_j$ has degree $j-d$ with leading coefficient $\frac{j !}{(j-d) !} 2^m T$.\\
Then, let $K_{s, d}(P, Q ; \boldsymbol{\Psi} ; q)$ be the number of solutions of$$
\sum_{i=1}^{k-d}\left(\Psi_j\left(z_i\right)-\Psi_j\left(w_i\right)\right)+q^j \sum_{i=1}^s\left(x_i^j-y_i^j\right)=0 \quad(1 \leq j \leq k),
$$with $1 \leq z_i, w_i \leq P $ and $ 1 \leq x_i, y_i \leq Q$. Also, let $L_{s, d}(P, Q ; \boldsymbol{\Psi} ; p, q, r)$ be the number of solutions of $$\sum_{i=1}^{k-d}\left(\Psi_j\left(z_i\right)-\Psi_j\left(w_i\right)\right)+(p q)^j \sum_{i=1}^s\left(u_i^j-v_i^j\right)=0 \quad(1 \leq j \leq k) $$with $1 \leq z_i, w_i \leq P$, $ z_i \equiv w_i\left(\bmod\  p^r\right)$ and $ 1 \leq u_i, v_i \leq Q$.
\begin{lemma}[\cite{ford_vinogradovs_2002} Lemma 3.2']\label{lemma3.2p}Suppose $k, r, d$ and $s$ are integers with
$$
k \ge 4 ; \quad 2 \le r \le k ; \quad 0 \le d \le r-1 ; \quad s \ge d+1 .
$$
Let $M, P$ and $Q$ be real numbers with
$$
P^{1 /(k+1)} \le M \le P^{1 / r} ; \quad 32 s^2 M<Q \le P, \quad M \ge k.
$$
Suppose $q$ is a positive integer and $\boldsymbol{\Psi}$ is a system of polynomials of type $(d, T)$ with $T \le P^d$. Denote by $\mathscr{P}$ the set of the $k^3$ smallest primes greater than $M$, and suppose $\mathscr{P} \subset(M, 2 M]$. Then there are a system of polynomials $\boldsymbol{\Phi}$ of type $(d, T)$ and a prime $p \in \mathscr{P}$ such that
    $$
K_{s, d}(P, Q ; \boldsymbol{\Psi} ; q) \leq 4 k^3 k ! p^{2 s+(r-d)(r-d+1) / 2} L_{s, d}(P, Q ; \boldsymbol{\Phi} ; p, q, r).
$$
\end{lemma}
\begin{lemma}[\cite{ford_vinogradovs_2002} Lemma 3.3']\label{lemma3.3p}Suppose that $s \geq d, k \geq r \geq 2, d \leq k-2, q \geq 1, $ $p$ is a prime and $\boldsymbol{\Phi}$ is a system of polynomials of type $(d, T)$. Then there is a system of polynomials $\boldsymbol{\Upsilon}$ of type $\left(d+1, T^{\prime}\right)$ with $T \leq T^{\prime} \leq P T$ such that
\begin{align*}
    L_{s, d}(P ; Q ; \boldsymbol{\Phi} ; p, q, r)&\le(2 P)^{k-d} \max \left[k^{k-d} J_{s, k}(Q), 2 p^{-r(k-d)} J_{s, k}(Q)^{(k-d-2) /(2(k-d-1))}\right.\\&\left.\times K_{s, d+1}(P, Q ; \boldsymbol{\Upsilon} ; p q)^{(k-d) /(2(k-d-1))}\right].
\end{align*}

\end{lemma}
As per Ford \cite{ford_vinogradovs_2002}, Lemma \ref{lemma3.2p} and Lemma \ref{lemma3.3p} imply the following result.
\begin{lemma}\label{lemma3.4p}
   Suppose $k \geq 26,4 \leq r \leq k, k \leq s \leq k^3$ and
$$
J_{s, k}(Q) \leq C Q^{2 s-k(k+1) / 2+\Delta} \quad(Q \geq 1) .
$$
Let $j$ be an integer satisfying
\begin{equation}\label{condj}
    2 \leq j \leq \frac{9}{10} r, \quad(j-1)(j-2) \leq 2 \Delta-(k-r)(k-r+1) .
\end{equation}

Define $$
\phi_J=\frac{1}{2 r}+\frac{k^2+k+r^2-r-2 \Delta-2 r J}{4 r(k-J)} \phi_{J+1} \quad(1 \leq J \leq j-1),
$$
and suppose $r$ and $j$ are chosen so that $\phi_i \geq 1 /(k+1)$ for every $i$. Suppose
$$
\frac{1}{3 \log k} \leq \omega \leq \frac{1}{2}, \quad \eta=1+\omega, \quad V=\max \left(e^{1.5+1.5 / \omega}, \frac{18}{\omega} k^3 \log k\right).
$$
If $P \geq V^{k+1}$, then
$$
J_{s+k, k}(P) \leq k^{3 k} \eta^{4 s+k^2} C P^{2(s+k)-k(k+1) / 2+\Delta^{\prime}}
$$
where $$\Delta^{\prime}=\Delta\left(1-\phi_1\right)-k+\frac{1}{2} \phi_1\left(k^2+k+r^2-r\right).$$
\end{lemma}
We will use the definition of $\phi_J$ given in Lemma \ref{lemma3.4p} for the range $k< 90000$. However, when $k$ becomes large, the influence that the new funcitons $\Phi_j$ have on estimating the Vinogradov's integral becomes negligible, as already pointed out by Ford in \cite{ford_vinogradovs_2002} (p. 590). Hence, for $k\ge 90000$, we use the following result used by Ford in his work.
\begin{lemma}[\cite{ford_vinogradovs_2002} Lemma 3.4]\label{lemma3.4f}
       Suppose $k \geq 26,4 \leq r \leq k, k \leq s \leq k^3$ and
$$
J_{s, k}(Q) \leq C Q^{2 s-k(k+1) / 2+\Delta} \quad(Q \geq 1) .
$$
Let $j$ be an integer satisfying the same relations as in \eqref{condj}. Then, define $$
\phi_j=\frac{1}{r}, \quad \phi_J=\frac{1}{2 r}+\frac{k^2+k+r^2-r+J^2-J-2 \Delta}{4 k r} \phi_{J+1} \quad(1 \leq J \leq j-1),
$$
and suppose $r$ and $j$ are chosen so that $\phi_i \geq 1 /(k+1)$ for every $i$. Suppose
$$
\frac{1}{3 \log k} \leq \omega \leq \frac{1}{2}, \quad \eta=1+\omega, \quad V=\max \left(e^{1.5+1.5 / \omega}, \frac{18}{\omega} k^3 \log k\right).
$$
If $P \geq V^{k+1}$, then
$$
J_{s+k, k}(P) \leq k^{3 k} \eta^{4 s+k^2} C P^{2(s+k)-k(k+1) / 2+\Delta^{\prime}}
$$
where $$\Delta^{\prime}=\Delta\left(1-\phi_1\right)-k+\frac{1}{2} \phi_1\left(k^2+k+r^2-r\right).$$
\end{lemma}
For a given $k, r$ and $\Delta$, we let $\delta_0(k, r, \Delta)$ be the value of $\Delta^{\prime}$ coming from Lemma \ref{lemma3.4p} or Lemma \ref{lemma3.4f}, where we take $j$ maximal satisfying \eqref{condj}.
\begin{lemma}[\cite{ford_vinogradovs_2002} Lemma 3.5]\label{lemma3.5f}
    Let $k \geq 26$ and let $\omega, \eta$ and $V$ be as in Lemma \ref{lemma3.4p} or Lemma \ref{lemma3.4f}. Let $\Delta_1=\frac{1}{2} k^2(1-1 / k)$ and for $n \geq 1$, let $r_n$ be an integer in $[4, k]$ satisfying
$$
\phi^*\left(k, r_n, \Delta_n\right):=\frac{2 k}{2 r_n k+2 \Delta_n-\left(k-r_n\right)\left(k-r_n+1\right)} \geq \frac{1}{k+1} .
$$
Then set $\Delta_{n+1}=\delta_0\left(k, r_n, \Delta_n\right)$. If $n \leq k^2$, then
$$
J_{n k, k}(P) \leq C_n P^{2 n k-k(k+1) / 2+\Delta_n} \quad(P \geq 1),
$$
where $C_1=k !$ and, for $n \geq 2$,
$$
C_n=C_{n-1} \max \left[k^{3 k} \eta^{4 k(n-1)+k^2}, V^{(k+1)\left(\Delta_{n-1}-\Delta_n\right)}\right].
$$
\end{lemma}
As already mentioned before, some explicit bounds for the Vinogradov's integral were found by Tyrina \cite{tyr87}, and they give evident improvements when $s\ll k^2$. In her argument, Tyrina defines recursively sequences $r_n, s_n$, and $\Delta_{s_n}$, with $ s_1=k$ and $\Delta_{s_1}=k$. More precisely, for $s \geq 1$ we take $r_n$ to be the integer nearest to the number
$$
\frac{2 \Delta_{s_n}+k(k+1)}{ 2 k+1}$$ and we define recursively $$
\Delta_{s_{n+1}}=\Delta_{s_n}+r_n-\frac{\left(2 r_n-k-1\right)\left(2 r_n-k\right)}{2 r_n}-\frac{\Delta_{s_n}}{r_n} $$where$$
s_{n+1}=k+r_1+\cdots+r_n.
$$
Then, we consider the function $x(y),\  k \leq y<k(k+1) / 2$ :
$$
x(y)=-y+2 k-(k+1)^2 \cdot \ln \left(1-2 \cdot(y-k) /\left(k^2-k\right)\right).
$$
This function increases monotonically and assumes the values from $k$ to infinity as $y$ increases from $k$ to $k(k+1) / 2$. The function $y=y(x)$ which is inverse to $x(y)$ is also monotonically increasing. Finally, let $l_0$ denote the integer $\lceil x\left(k^2 / 2\right)\rceil$, where
$$
x\left(k^2 / 2\right)=-k^2 / 2+2 k+(k+1)^2 \ln (k-1).
$$

Tyrina proved the following result.
\begin{theorem}[\cite{tyr87} Thm. 1]\label{tyrina}
     The mean value $J_{s, k}(P)$  satisfies the estimate
$$
J_{s, k}(P) \leq D_s \cdot P^{2 s-\kappa_s}
$$
where $D_s=2^{2 s\left(k+s k^{-1}\right)} k^{k+4 s^2} s^{2(s-k)}$. Here
\begin{enumerate}
    \item if $s=s_n, n=1,2, \ldots, s_n<l_0$, then $\kappa_s \geq \Delta_s \geq y(s)$;
    \item if $s=l_t=l_0+k t, t=0,1, \ldots$, then
$$
\kappa_s \geq \Delta_s \geq \frac{k(k+1)}{2}-\frac{k+2}{2}\left(1-\frac{1}{k}\right);
$$
\item if $s_n<s<s_{n+1}$, then $\kappa_s \geq \Delta_s \geq \Delta_{s_{n+1}} \cdot s / s_{n+1}$; 
\item  if $l_{t-1}<s<l_t, t=1,2, \ldots$, then $\kappa_s \geq \Delta_s \geq \Delta_{l_t} \cdot s / l_t$.
\end{enumerate}
\end{theorem}
New explicit bounds for $J_{s,k}$ were found by Ford in 2002 \cite{ford_vinogradovs_2002}.
\begin{theorem}[\cite{ford_vinogradovs_2002} Thm. 3]\label{th3ford}
 Let $k$ and $s$ be integers with $k \geq 1000$ and $2 k^2 \leq s \leq \frac{k^2}{2}\left(\frac{1}{2}+\log \frac{3 k}{8}\right)$. Then
$$
J_{s, k}(P) \leq k^{2.055 k^3-5.91 k^2+3 s} 1.06^{s k+2 s^2 / k-9.7278 k^3} P^{2 s-\frac{1}{2} k(k+1)+\Delta_s} \quad(P \geq 1)
$$
where
$$
\Delta_s=\frac{3}{8} k^2 e^{1 / 2-2 s / k^2+1.7 / k}.
$$
Further, if $k \geq 129$, there is an integer $s \leq \rho k^2$ such that for $P \geq 1$,
$$
J_{s, k}(P) \leq k^{\theta k^3} P^{2 s-\frac{1}{2} k(k+1)+0.001 k^2}
$$
with
\begin{equation}\label{fordrho}
    (\rho, \theta)=\left\{\begin{array}{ll}
(3.21432,2.3291) & (k \geq 200) \\ \\
(3.21734,2.3849) & (150 \leq k \leq 199) \\ \\
(3.22313,2.4183) & (129 \leq k \leq 149)
\end{array}\right. .
\end{equation}
\end{theorem}
In 2011, Preobrazhenski\u{\i} \cite{preobrazhenskii_new_2011} improved the value of  $\rho$ in Theorem \ref{th3ford} when $k\ge 16000$ to $\rho=3.213303$. Some further explicit bounds for the Vinogradov's integral were found also by Steiner \cite{steiner_new_2019} using Wooley's efficient congruencing method \cite{wooley2012vinogradov, wooley2013, Wooley2017NestedEC}. He proved that for every $k\ge 3$, $s\ge \frac{5}{2}k^2+k$, and $P\ge s^{10}$, we have $J_{s,k}(P)\le CP^{2s-\frac{1}{2}k(k+1)}$, where $C$ is roughly $k^{k^{O(\log(k)/\log(\lambda))}}$. However, although the exponent of $P$ is the optimal one, the constant $C$ is far too large compared to that one found by Ford in Theorem \ref{th3ford}, which is of order $k^{O(k^3)}$. In order to have a reasonable estimate for $J_{s,k}$, one should find a constant that is at most of order $k^{O(k^4)}$, since otherwise, after having taken the $k^4-$th root (this passage is required by the argument in Section \ref{sect4}, where the constant appearing in the upper bound for the Vinogradov integral will be elevated to the power of $1/(2rs)$, with both $r$ and $s$ of order $O(k^2)$), the constant cannot be controlled. Hence, as already Steiner pointed out in his paper (\cite{steiner_new_2019}, p. 359), the estimate he found for $J_{s,k}$ cannot be applied to Ford's argument.\\

In \cite{ford_vinogradovs_2002}, Ford found also an upper bound for the incomplete systems defined by \eqref{defincomplete}.
\begin{theorem}[\cite{ford_vinogradovs_2002} Thm.4]\label{th4ford}
Let $k \ge 60$, $h \in [0.9k, k - 2]$, $s \in [2t, \lfloor h/2\rfloor t]$ and $P \ge \exp(Dk^2)$ for some $D \ge 10$. Assume that 
\[
\eta \in \left(2k^{-3}, (2k)^{-1}\right],\qquad \frac{4\log k}{Dk^2\eta} \in \left[18k^{-1}, 0.4\right].
\]
Then 
\[
J_{s, k, h}(\mathcal{C}(P, P^\eta)) \le A P^{E}
\]
where
\[
A := \exp\left(\frac{s^2}{t} + \frac{10.5t\log^2 k}{Dk\eta^2} + s\left(\left(\eta^{-1} + h\right)\left(1 - h^{-1}\right)^{s/t} - h\right)\log(10\eta)\right),
\]
\[
E := 2s - \frac{t(h + k)}{2} + \frac{t(t - 1)}{2} + \frac{\eta s^2}{2t} + ht\exp\left(-\frac{s}{ht}\right).
\]
\end{theorem}
Finally, we recall a few results found in \cite{ford_vinogradovs_2002} involving both $S(N,t)$ defined in \eqref{Snt} and some bounds for both $\zeta(s)$ and $\zeta(s,u)$ that we will use to prove Theorem \ref{theorem1} and Theorem \ref{theorem2} respectively.
\begin{lemma}[\cite{ford_vinogradovs_2002} Lemma 5.1]\label{lemma5.1ford}
Suppose $k, r$ and $s$ are integers $\geq 2$, and $h$ and $g$ are integers satisfying $1 \leq h \leq g \leq k$. Let $N$ be a positive integer, and $M_1, M_2$ be real numbers with $1 \leq M_i \leq N$. Let $\mathscr{B}$ be a nonempty subset of the positive integers $\leq M_2$. Then
\begin{align*}
    &S(N,t)\leq 2M_1 M_2+\frac{t(M_1 M_2)^{k+1}}{(k+1)N^k}\\&+N\left(\frac{M_2}{|\mathscr{B}|}\right)^{1/r}\left((5r)^kM_2^{-2s}\lfloor M_1\rfloor ^{-2r+k(k+1)/2}J_{r,k}(\lfloor M_1\rfloor)J_{s,g,h}(\mathscr{B})W_h\dots W_g\right)^{1/2rs},
\end{align*}
    where
$$
W_j=\min \left(2 s M_2^j, \frac{2 s M_2^j}{r\left\lfloor M_1\right\rfloor^j}+\frac{s t M_2^j}{\pi j N^j}+\frac{4 \pi j(2 N)^j}{r t\left\lfloor M_1\right\rfloor^j}+2\right) \quad(j \geq 1) .
$$
\end{lemma}
\begin{lemma}[\cite{ford_vinogradovs_2002}, Lemma 7.1]\label{lemma7.1}
Suppose $\frac{1}{2} \leq \sigma \leq 1,0<u \leq 1, t \geq 3$ and $s=\sigma+i t$. If either $\sigma \leq \frac{15}{16}$ or $t \leq 10^{100}$, then
$$
|\zeta(s)|,\left|\zeta(s, u)-u^{-s}\right| \leq 58.1 t^{4(1-\sigma)^{3 / 2}} \log ^{2 / 3} t.
$$
\end{lemma}
\begin{lemma}[\cite{ford_vinogradovs_2002}, Lemma 7.2]\label{lemma7.2}
If $s=\sigma+i t, \frac{15}{16} \leq \sigma \leq 1, t \geq 10^{100}$ and $0<u \leq 1$, then
$$
\left|\zeta(s, u)-\sum_{0 \leq n \leq t}(n+u)^{-s}\right| \leq 10^{-80}.
$$
\end{lemma}
\begin{lemma}[\cite{ford_vinogradovs_2002}, Lemma 7.3]\label{lemma73}
Suppose that $S(N, t) \leq C N^{1-1 /\left(D \lambda^2\right)}(1 \leq N \leq t)$ for positive constants $C$ and $D$, where $\lambda=\frac{\log t}{\log N}$. Let $B=\frac{2}{9} \sqrt{3 D}$. Then, for $\frac{15}{16} \leq \sigma \leq 1, t \geq 10^{100}$ and $0<u \leq 1$, we have
\begin{equation*}
    \begin{aligned}
        |\zeta(s)| &\leq\left(\frac{C+1+10^{-80}}{\log ^{2 / 3} t}+1.569 C D^{1 / 3}\right) t^{B(1-\sigma)^{3 / 2}} \log ^{2 / 3} t \\
        \left|\zeta(s, u)-u^{-s}\right| &\leq\left(\frac{C+1+10^{-80}}{\log ^{2 / 3} t}+1.569 C D^{1 / 3}\right) t^{B(1-\sigma)^{3 / 2}} \log ^{2 / 3} t .
    \end{aligned}
\end{equation*}
\end{lemma}
\section{Proof of Theorem \ref{newboundsvino}}
Let $s\ll k^2$. Due to the size of $s$, we are in the first or third cases of Tyrina's Theorem \ref{tyrina}. However, since the third case gives an estimate for those $s$ between two consecutive $s_n$ for which we have an estimate in the first case, we can restrict ourselves to the first case. By Theorem \ref{tyrina}, we know that
\[
J_{s,k}\ll P^{2s-y(s)},
\]where $y(x)$ is the inverse function of $x(y)$, with 
\[
x(y)=-y+2 k-(k+1)^2 \log \left(1-\frac{2 (y-k)} {\left(k^2-k\right)}\right).
\]
A direct computation gives the following explicit expression for $y(x)$:
\[
 y(x) = \frac{k(k + 1)}{2} - (k + 1)^2 W\left(\frac{k(k - 1)}{2 (k + 1)^2}\exp\left(\frac{k}{2k + 2} + \frac{x}{(k+ 1)^2} - \frac{2k}{(k + 1)^2}\right)\right),
\]
where $W$ is the principal branch of the Lambert product-log function.\\Furthermore, the function $x(y)$ is monotonically increasing in the interval $\left[n, \frac{n(n + 1)}{2}\right)$. Therefore, to lower-bound $y(x)$, it suffices to upper-bound $x(y)$. We want to find $s$ such that $y(s)\ge 0.101k^2$. This means that $s\le x(0.101k^2)\le 0.1247k^2$ for every $k\ge 50$.\\
However, 
\begin{equation}\label{bound500}
    2s-0.101k^2\le 2s -\frac{k(k+1)}{2}+0.4k^2,\qquad \forall k\ge 500,\quad s\ge 0.1247k^2.
\end{equation}

It follows that, given 
\[
J_{s,k}(P)\ll P^{2s-\frac{k(k+1)}{2}+\Delta_s},
\]
we have 
\begin{equation}\label{del500}
    \Delta_s\le 0.4k^2\qquad \forall s\ge 0.1247k^2,\quad k\ge 500,
\end{equation}
since $\Delta_s$ is decreasing in $s$, for every fixed $k$.\\
In order to prove the first part of Theorem \ref{newboundsvino}, we follow Ford's method in \cite{ford_vinogradovs_2002} and we find an estimate for $J_{nk,k}$. By \eqref{del500}, we know that, given $k\ge 90000$ fixed, we have $\Delta_s\le 0.4k^2$ for every $s\ge 0.1247k^2$. Hence, in particular, if we denote $n_0=\lceil0.1247k\rceil$, we have $\Delta_{n_0}:=\Delta_{n_0k}\le 0.4k^2$. As a result, for the case $\Delta_{n-1}>k$, we follow Ford's method in Lemma 3.6 of \cite{ford_vinogradovs_2002} but instead of starting from $\Delta_1$, we start from $\Delta_{n_0}$.
\begin{lemma}\label{lemma36new}
    For every $1\le D\le 0.4k$, $k\ge 90000$ and
    \[
    0.138128 k\le n\le \frac{1}{2}k\left(0.6494+\log\left(\frac{8k}{25D}\right)-\left(\frac{1-D/k}{2-D/k}\right)\frac{D}{k}+\frac{2.051}{k}\right)+1
    \]
    we have
    \[J_{nk,k}(P)\le C_nP^{2nk-k(k+1)/2+\Delta_n}\]
    with \[
\Delta_n\le \max\left(\frac{8}{25}k^2 e^{0.6494-2 n / k-\left(\frac{1-D/k}{2-D/k}\right)\frac{D}{k}+2.051 / k},Dk\right)
\]and
\[
C_n\le k^{2.055 k^3-0.414 k^2+3 n k} 1.06^{n k^2+2\left(n^2-n\right) k+0.099912 k^3}.
\]
\end{lemma}
\begin{proof}
 We take 
    \[
    r_n=\left\lfloor \sqrt{k^2+k-2\Delta_n}\right\rfloor
    \]in Lemma \ref{lemma3.5f}, where $\phi_J$ are defined as in Lemma \ref{lemma3.4f}. Then, we define $\delta_n=\Delta_n/k^2$ for every $n$. Now, we fix $n\geq 2$ and we define $\delta=\delta_{n-1}$, $\delta'=\delta_n$, $\Delta=\Delta_{n-1}$, $\Delta'=\Delta_n$ and $r=r_{n-1}$.\\
    If $\Delta_{n-1}\le k$, then the upper bound for $\Delta_n$ follows trivially, since \[
\Delta_n\le \Delta_{n-1}\le k\le Dk \qquad \forall 1\le D\le 0.4k,\ k\ge 90000.
\] 
  Hence, from now on, we just consider $\Delta_{n-1}>k$.\\
By \eqref{del500}, we know that, given $k\ge 90000$ fixed, we have $\Delta_s\le 0.4k^2$ for every $s\ge 0.1247k^2$. Hence, in particular, if we denote $n_0=\lceil0.1247k\rceil$, we have $\Delta_{n_0}:=\Delta_{n_0k}\le 0.4k^2$. At this point, we follow Ford's method in Lemma 3.6 of \cite{ford_vinogradovs_2002} but instead of starting from $\Delta_1$, we start from $\Delta_{n_0}$.\\
    Let
    \[
    y=2\Delta-(k-r)(k-r+1),\qquad \phi^*=\phi^*(k,r,\Delta)=\frac{2k}{2rk+y}.
    \]
    Using the definition of $r_n$, we have
    \[
    \sqrt{k^2+k-2\Delta}-1\le r=r_{n-1}=\lfloor \sqrt{k^2+k-2\Delta}\rfloor\leq  \sqrt{k^2+k-2\Delta}.
    \]
    Hence,
     \begin{equation*}
        \begin{aligned}
            \phi^*=\phi^*(k, r, \Delta)&=\frac{2 k}{2 r k+y}\\&\geq \frac{2k}{2k\sqrt{k^2+k-2\delta k^2}+\delta k\left(2k-1+\frac{1}{2\delta k}\right)}\\ &=\frac{2k}{2k^2\sqrt{1+\frac{1}{k}-2\delta}+\delta k\left(2k-1+\frac{1}{2\delta k}\right)}\\&\ge \frac{2k}{2k^2\left(1+\frac{1}{2k}-\delta\right)+\delta k\left(2k-1+\frac{1}{2\delta k}\right)}\\&= \frac{2}{2k+1-2\delta k+2\delta k-\delta +\frac{1}{2k}}\\&=\frac{2}{2k+1-\delta+\frac{1}{2k}}\\&\ge \frac{1}{k+1}
        \end{aligned}
    \end{equation*}
    and so the hypotheses of Lemma \ref{lemma3.5f} are satisfied.\\As per Ford \cite{ford_vinogradovs_2002}, we have
    \begin{equation}
\theta_1\le \frac{2^{1-j}}{r}+\frac{2 \phi^*}{k r} \leq \frac{0.071}{k^4 r}+\frac{2 \phi^*}{k r}.
\end{equation}
Since $\delta\le 0.4$, denoting with
\[\alpha=\sqrt{k^2+k-2\delta k^2}-1,\]
we have
\begin{equation}
    \begin{aligned}
        \phi^*&=\frac{2k}{2rk+y}\\&\le \frac{2k}{2k\alpha+2\delta k^2-(k-\alpha)(k-\alpha+1)}\\&\le \frac{2}{\left(2-\delta^2\right) k-\delta}-\frac{1}{k} + \frac{2k}{\sqrt{2k}(3 + 4 k) + 2 (-1 - 3 k - k^2)}\\&<\frac{1.08696}{ k}+\frac{0.2363}{k^2}-\frac{1}{k} + \frac{2k}{\sqrt{2k}(3 + 4 k) + 2 (-1 - 3 k - k^2)}\\&\le \frac{1.08696}{ k}+\frac{0.2363}{k^2}-\frac{1}{k}\\&=\frac{0.08696}{ k}+\frac{0.2363}{k^2}
    \end{aligned}
\end{equation}as
\[
\frac{2k}{\sqrt{2k}(3 + 4 k) + 2 (-1 - 3 k - k^2)}<0\qquad \forall k\ge 90000.
\]
It follows that 
\[
\theta_1\le \frac{0.071}{k^4 r}+\frac{0.08696}{ k^2 r}+\frac{0.2363}{k^3 r} \leq\frac{0.08697}{ k^2 r}.\]
As per Ford \cite{ford_vinogradovs_2002},
   \begin{equation}\label{eq1D}
       \begin{aligned}
           \Delta^{\prime}&=\Delta-k+\frac{\phi^*+\theta_1}{2}(2 k r-y)\\&\leq \Delta-k+\frac{1}{2}\phi^*(2kr-y)+\frac{0.043485}{k^2}\left(2k-\frac{y}{r}\right)\\&=\Delta-k+\frac{1}{2}\phi^*(2kr-y)+\frac{0.043485}{k^2}\left(2k-\frac{1}{r}\left(2\Delta-(k-r)(k-r+1)\right)\right)\\&\le \Delta-k+\frac{1}{2}\phi^*(2kr-y)+\frac{0.043485}{k^2}\left(r-1+\frac{k^2+k-2\Delta}{r}\right)\\&=\Delta-\frac{2ky}{2rk+y}+\frac{0.043485}{k^2}\left(r-1+\frac{k^2+k-2\Delta}{r}\right)\\&=\Delta-2k+4k^2\frac{r}{2rk+y}+\frac{0.043485}{k^2}\left(r-1+\frac{k^2+k-2\Delta}{r}\right)
       \end{aligned}
   \end{equation}
because
   \begin{equation*}
       \begin{aligned}
           -k+\frac{1}{2}\phi^*(2kr-y)&=-k+\frac{1}{2}\frac{2k}{2rk+y}(2kr-y)=-k+k\left(1-\frac{2y}{2kr+y}\right)=-\frac{2ky}{2rk+y}.
       \end{aligned}
   \end{equation*} As a function of the real variable $r$, \[
   \frac{r}{2rk+y}
   \]has a positive second derivative and a minimum at 
   \[
   r_0=\sqrt{k^2+k-2\Delta}.
   \]It follows that in the interval \[
    \sqrt{k^2+k-2\Delta_{n-1}}-1\le r\le  \sqrt{k^2+k-2\Delta}
    \]the maximum occurs in 
    \[
    r_1= \sqrt{k^2+k-2\Delta_{n-1}}-1.
    \]
  Hence,
  \begin{equation}\label{lowerr}
      \begin{aligned}
          &\frac{r}{2rk+y}\ge\frac{r_0}{2r_0k+y}\\&= \frac{\sqrt{k^2+k-2\delta k^2}}{2k(\sqrt{k^2+k-2\delta k^2})+2\delta k^2-(k-\sqrt{k^2+k-2\delta k^2})(k-\sqrt{k^2+k-2\delta k^2}+1)}\\&\ge \frac{1-\delta}{2k}
      \end{aligned}
  \end{equation}
and
\begin{align*}
  & \frac{r}{2rk+y}\le \frac{r_1}{2r_1k+y}\\&= \frac{\sqrt{k^2+k-2\delta k^2}-1}{2k(\sqrt{k^2+k-2\delta k^2}-1)+2\delta k^2-(k-\sqrt{k^2+k-2\delta k^2}+1)(k-\sqrt{k^2+k-2\delta k^2}+1+1)}. 
\end{align*}
We want to estimate the quantity $$\frac{r_1}{2r_1k+y}$$ using Taylor expansion for $k\rightarrow +\infty$.
We observe that
\begin{align*}
    &\lim_{k\rightarrow\infty} k\cdot  \frac{r_1}{2r_1k+y}=\frac{\sqrt{1 - 2\delta}}{-2 + 4\sqrt{1 - 2\delta} + 4\delta}.
\end{align*}
Furthermore, for all $\delta \in (0, 0.4)$ we have
\[
\frac{\sqrt{1 - 2\delta}}{-2 + 4\sqrt{1 - 2\delta} + 4\delta} < \frac{1 - \delta}{2 - \delta^2}.
\]
 Now,
\[
\frac{\sqrt{1 - 2\delta}}{-2 + 4\sqrt{1 - 2\delta} + 4\delta}=\frac{1}{2} - \frac{\delta}{2} + \frac{\delta^2}{4} - \frac{\delta^3}{4} + \frac{\delta^4}{16}-\frac{3\delta^5}{16}+O(\delta^6)
\]
and
\[
\frac{1 - \delta}{2 - \delta^2} =\frac{1}{2} - \frac{\delta}{2} + \frac{\delta^2}{4} - \frac{\delta^3}{4} + \frac{\delta^4}{8}-\frac{\delta^5}{8}+O(\delta^6).
\]It follows that
\begin{equation}\label{firsterm}
    \frac{\sqrt{1 - 2\delta}}{-2 + 4\sqrt{1 - 2\delta} + 4\delta}\frac{1}{k} \le \left(\frac{1 - \delta}{2 - \delta^2}-\frac{\delta^4}{16}-\frac{\delta^5}{16}\right)\frac{1}{k}.
\end{equation}
Now, we repeat the process for the term of second order. We have
\begin{align*}
    &\lim_{k\rightarrow\infty} k^2\cdot\left(  \frac{r_1}{2r_1k+y}-\frac{\sqrt{1 - 2\delta}}{-2 + 4\sqrt{1 - 2\delta} + 4\delta}\frac{1}{k}\right)=\frac{-1 + \sqrt{1 - 2 \delta} + 2 \delta}{4 (-1 + 2 \sqrt{1 - 2 \delta} + 2 \delta)^2}.
\end{align*}
As before, using the Taylor expansion in $\delta=0$ ($\delta>1/k\rightarrow0$ for $k\rightarrow +\infty$), we have\begin{equation}
    \frac{-1 + \sqrt{1 - 2 \delta} + 2 \delta}{4 (-1 + 2 \sqrt{1 - 2 \delta} + 2 \delta)^2}\frac{1}{k^2}\le
\left(\frac{4\delta}{(2-\delta^2)^2}-\frac{3\delta}{4}\right)\frac{1}{k^2}.
\end{equation}However, in order to have an upper bound for \[
\frac{r_1}{2r_1k+y}
\]a direct computation shows that we need to use the slightly weaker bound\begin{equation}\label{secondterm}
    \frac{-1 + \sqrt{1 - 2 \delta} + 2 \delta}{4 (-1 + 2 \sqrt{1 - 2 \delta} + 2 \delta)^2}\frac{1}{k^2}\le
\left(\frac{4\delta}{(2-\delta^2)^2}-\frac{14\delta}{25}\right)\frac{1}{k^2},
\end{equation}where we have $-14\delta/25$ instead of $-3\delta/4$. Using the upper bounds found in \eqref{firsterm} and \eqref{secondterm} we obtain
\begin{equation}\label{uppphi}
    \frac{r}{2rk+y}\le \left(\frac{1 - \delta}{2 - \delta^2}-\frac{\delta^4}{16}-\frac{\delta^5}{16}\right)\frac{1}{k}+\left(\frac{4\delta}{(2-\delta^2)^2}-\frac{14\delta}{25}\right)\frac{1}{k^2}.
\end{equation}
From \eqref{eq1D} and \eqref{uppphi}, it follows that
\begin{equation*}
\begin{aligned}
    \delta' &\le \delta -\frac{2}{k} + 4\left(\frac{1-\delta}{2-\delta^2} - \frac{\delta^4}{16}- \frac{\delta^5}{16}\right)\frac{1}{k}+\left(\frac{4\delta}{(2-\delta^2)^2}-\frac{14\delta}{25}\right)\frac{1}{k^2}\\&\ \ +\frac{0.043485}{ k^4}\left(r-1+\frac{k^2+k-2 \Delta}{r}\right)\\&= \delta\left(1 - \frac{4 - 2\delta}{(2-\delta^2)k} - \frac{\delta^3}{4k}- \frac{\delta^4}{4k}+\frac{4}{(2-\delta^2)^2 k^2}-\frac{14}{25k^2}\right)\\&\ \ +\frac{0.043485}{k^4}\left(\sqrt{k^2+k-2\delta k^2}-1+\frac{k^2+k-2 \delta k^2}{\sqrt{k^2+k-2\delta k^2}-1}\right)\\&\le\delta\left(1 - \frac{4 - 2\delta}{(2-\delta^2)k} - \frac{\delta^3}{4k}- \frac{\delta^4}{4k}+\frac{4}{(2-\delta^2)^2 k^2}-\frac{14}{25k^2}\right)+\frac{0.08697(1-\delta)}{ k^3}\\&\le \delta\left(1 - \frac{2-\delta}{2-\delta^2}\left(\frac{2}{k} +\frac{\delta^3}{4k}+ \frac{\delta^4}{4k}-\frac{32}{21 k^2}+\frac{14}{25k^2}\right)\right)+\frac{0.08697}{k^3}\frac{2-\delta}{2-\delta^2}\\&=\delta\left(1 - \frac{2-\delta}{2-\delta^2}\left(\frac{2}{k} +\frac{\delta^3}{4k}+ \frac{\delta^4}{4k}-\frac{506}{525 k^2}-\frac{0.08697}{ k^3\delta}\right)\right)\\&\le\delta\left(1 - \frac{2-\delta}{2-\delta^2}\left(\frac{2}{k} +\frac{1}{4k^4}+ \frac{1}{4k^5}-\frac{506}{525 k^2}-\frac{0.08697}{ k^3\delta}\right)\right),
\end{aligned} 
\end{equation*}where in the last passage we used the fact that $\delta>1/k$.\\
Now we define \[
\beta=\frac{2}{k}-\frac{506}{525k^2}+\frac{1}{4k^4}+ \frac{1}{4k^5},\qquad \beta^\prime=\beta-\frac{c}{\delta},\qquad c=\frac{0.08697}{ k^3}
\]and let$$
\delta^{\prime \prime}=\delta\left(1-\frac{2-\delta}{2-\delta^2} \beta^{\prime}\right).
$$Since $y+\log y+\log (2-y)$ is increasing on $\left(0, \frac{1}{2}\right]$, and hence in $\left(0,0.4\right]$, we have\begin{equation*}
\begin{aligned}
\delta^{\prime}+\log \delta^{\prime}+\log \left(2-\delta^{\prime}\right) \leq & \delta^{\prime \prime}+\log \delta^{\prime \prime}+\log \left(2-\delta^{\prime \prime}\right) \\
= & \delta+\log \delta+\log (2-\delta) \\
& -\frac{2 \delta-\delta^2}{2-\delta^2} \beta^{\prime}+\log \left[\left(1-\frac{2-\delta}{2-\delta^2} \beta^{\prime}\right)\left(\frac{2-\delta^{\prime \prime}}{2-\delta}\right)\right].
\end{aligned}
\end{equation*}As per Ford \cite{ford_vinogradovs_2002}, given\begin{equation*}
T=-\frac{2 \delta-\delta^2}{2-\delta^2} \beta^{\prime}+\log \left(1-\frac{2-\delta}{2-\delta^2} \beta^{\prime}\right)+\log \left(\frac{2-\delta^{\prime \prime}}{2-\delta}\right),
\end{equation*}we have
\begin{equation*}
\begin{aligned}
T & \leq-\beta^{\prime}-\frac{\left(\beta^{\prime}\right)^2}{2\left(2-\delta^2\right)^2}\left((2-\delta)^2+\delta^2\right)+\frac{\left(\beta^{\prime}\right)^3}{3\left(2-\delta^2\right)^3}\left(-(2-\delta)^3+\delta^3\right) \\
& \leq-\beta^{\prime}-\frac{2}{5}\left(\beta^{\prime}\right)^2 \\
& \leq-\beta-\frac{2}{5} \beta^2+\frac{c(1+0.8 \beta)}{\delta} .
\end{aligned}
\end{equation*}Hence, using an iterative argument we get
\begin{equation*}
\begin{aligned}
\delta_n+\log \delta_n+\log \left(2-\delta_n\right) \leq & \delta_{n_0}+\log \delta_{n_0}+\log \left(2-\delta_{n_0}\right)-(n-n_0)\left(\beta+0.4 \beta^2\right) \\
& +c(1+1.6 / k)\left(\frac{1}{\delta_{n_0}}+\ldots+\frac{1}{\delta_{n-1}}\right) .
\end{aligned}
\end{equation*}Since we are working with $\Delta_n>k$, that is $\delta>1/k$, and we found the inequality 
\begin{equation}\label{ineqdelta}
    \delta^\prime\le \delta\left(1 - \frac{2-\delta}{2-\delta^2}\left(\frac{2}{k} +\frac{1}{4k^4}+ \frac{1}{4k^5}-\frac{506}{525 k^2}-\frac{0.08697}{k^3\delta}\right)\right),
\end{equation}
we have
\begin{equation*}
\delta_{i+1} \leq \delta_i(1-\alpha), \quad \alpha=0.869565(\beta-k c).
\end{equation*}It follows that
\begin{equation*}
c(1+1.6 / k)\left(\frac{1}{\delta_{n_0}}+\ldots+\frac{1}{\delta_{n-1}}\right) \leq \frac{c(1+1.6 / k)}{\alpha \delta_{n-1}} \leq \frac{0.051}{k} .
\end{equation*}Therefore
\begin{equation*}
\delta_n \leq \frac{\delta_{n_0}\left(2-\delta_{n_0}\right) e^{\delta_{n_0}}}{\left(2-\delta_n\right) e^{\delta_n}} e^{-(n-n_0)\left(\beta+0.4 \beta^2\right)+0.051 / k}
\end{equation*}Now,
\begin{equation*}
\beta+0.4 \beta^2 \geq \frac{2}{k}-\frac{506}{525k^2}+\frac{1}{4k^4}+ \frac{1}{4k^5}+\frac{0.4}{k^2}\left(2-\frac{506}{525\cdot90000}\right)^2\ge \frac{2}{k} .
\end{equation*}
Furthermore
\begin{equation*}
\frac{e^{-\delta_n}}{2-\delta_n} \leq\frac{1}{2} e^{\delta_n /\left(2-\delta_n\right)-\delta_n}
\end{equation*}and, since
 $\delta_n=\Delta_n/k^2$, and $\delta_n\le0.4$ for $k\ge 90000$, if we write $\delta_n=D/k$, where $1\le D\le 0.4k$, we have  $$
\frac{e^{-\delta_n}}{2-\delta_n} \leq \frac{1}{2} e^{\delta_n /\left(2-\delta_n\right)-\delta_n}=\frac{1}{2}e^{-\left(\frac{1-D/k}{2-D/k}\right)\frac{D}{k}}.
$$

Also, using the fact that $\delta_{n_0}\le 0.4$, we have $$
\delta_{n_0}\left(2-\delta_{n_0}\right) e^{\delta_{n_0}}\le0.4\cdot \left(2-0.4\right)e^{0.4}=\frac{16}{25}e^{0.4}.
$$
Furthermore, for $k\ge 90000$, we have
\[
e^{2n_0/k}\le e^{0.2494+2/k}.
\]
It follows that
\begin{equation*}
\delta_n \leq \frac{8}{25}e^{0.6494-2 n / k-\left(\frac{1-D/k}{2-D/k}\right)\frac{D}{k}+2.051 / k}
\end{equation*}and so \[
\Delta_n\le \frac{8}{25}k^2 e^{0.6494-2 n / k-\left(\frac{1-D/k}{2-D/k}\right)\frac{D}{k}+2.051 / k}.
\]
Now we shift our attention on the constant. As per Ford \cite{ford_vinogradovs_2002}, to bound the constants $C_n$, we choose $\omega=0.06>1/(3\log k)$ in Lemma \ref{lemma3.5f}, so that
$$
V^{k+1}=\left(300 k^3 \log k\right)^{k+1} \leq k^{4.11 k}=: W 
$$
for every $k\ge 90000$.
Then, 
\begin{equation}\label{estw}
   W^{\Delta_{n-1}-\Delta_n}>k^{3 k} 1.06^{4 k(n-1)+k^2} \quad(n \leq 0.138 k+1),
\end{equation} since it was proved by Ford in \cite{ford_vinogradovs_2002} for the wider range $n\le 1.97k+1$.\\
Now, let $n_1=\lfloor 0.138 k\rfloor+1$. By \eqref{estw} and Lemma \ref{lemma3.5f},
$$
C_{n_1} \leq W^{\Delta_1-\Delta_{n_1}} k ! \leq W^{k^2 / 2-\Delta_{n_1}}
$$
and, for $n>n_1$,
$$
C_n \leq k^{3 k} 1.06^{4 k(n-1)+k^2} W^{\Delta_{n-1}-\Delta_n} C_{n-1} .
$$
Iterating this last inequality gives, for $n>n_1$,
$$
\begin{aligned}
C_n & \leq W^{k^2 / 2} k^{3 k\left(n-n_1\right)} 1.06^{\left(n-n_1\right) k^2+4 k\left(n_1+\ldots+n-1\right)} \\
& \leq W^{k^2 / 2} k^{3 k(n-0.138 k)} 1.06^{(n-0.138 k) k^2+2 k\left(n^2-n-(0.138 k)^2+0.138 k\right)} \\
& \leq k^{2.055 k^3-0.414 k^2+3 n k} 1.06^{n k^2+2\left(n^2-n\right) k+0.099912 k^3}.
\end{aligned}
$$
This concludes the proof of the lemma.
\end{proof}
\begin{remark}
    In Lemma \ref{lemma36new}, contrary to Lemma 3.6 in \cite{ford_vinogradovs_2002}, we are not allow to consider $2k$ as a lower bound for the range of $n$ for which the lemma holds. Indeed, the upper bound for $n$, that is $\frac{1}{2}k\left(0.6494+\log\left(\frac{8k}{25D}\right)-\left(\frac{1-D/k}{2-D/k}\right)\frac{D}{k}+\frac{2.051}{k}\right)+1$, is a decreasing function in $D$. As a result, this upper bound reaches the minimum when, for $k\ge 90000$ fixed, $D=0.4k$. Substituting $D=0.4k$ in the above expression, one gets $n\le 2.0255 + 0.138128 k$, which is less than $2k$ for $k\ge 90000$. Hence, in order to have an estimate for $\Delta_n$ and $C_n$ which is uniform for every $1\le D\le 0.4 k$ and $k\ge 90000$, we took $0.138128k$ as lower bound for $n$, in order to ensure that for every $D$ and $k$ there exists always at least a value of $n$ in the range given in the hypotheses of Lemma \ref{lemma36new}. Also, the choice $0.138128k$ as lower bound for $n$ is admissible, as in Lemma \ref{lemma36new} the starting point is $n_0=\lceil0.1247k\rceil$ which is less than $0.138128k$ for $k\ge 90000$.
\end{remark}
Now, we will use Lemma \ref{lemma36new} to prove the first part of the theorem.\\Given $k\ge 90000$, every admissible $s$ such that $s \not \equiv 0\ (\bmod\  k)$, is of the form $s=n k+u$, where $0<u<k, \ u\equiv s \bmod k$, and
\[
 0.138128k\le n\le \frac{1}{2}k\left(0.6494+\log\left(\frac{8k}{25D}\right)-\left(\frac{1-D/k}{2-D/k}\right)\frac{D}{k}+\frac{2.051}{k}\right).
\]
Using Hölder's inequality, we get
$$
J_{s, k}(P)=\int_{[0,1]^k}|S|^{2 n k+2 u} d \boldsymbol{\alpha} \leq\left(\int_{[0,1]^k}|S|^{2 n k} d \boldsymbol{\alpha}\right)^{1-u / k}\left(\int_{[0,1]^k}|S|^{2 n k}|S|^{2 k} d \boldsymbol{\alpha}\right)^{u / k},
$$where $\boldsymbol{\alpha}=(\alpha_1,\dots,\alpha_k)$ and
\[
S:=\sum_{1 \leq x \leq P} e\left(\alpha_1 x+\cdots+\alpha_k x^k\right).
\]
By Lemma \ref{lemma36new}, we have
$$
J_{s, k}(P) \leq   k^{2.055 k^3-0.414 k^2+3 n k} 1.06^{n k^2+2\left(n^2-n\right) k+0.099912 k^3} P^{2 s-k(k+1) / 2+\Delta}
$$
where
 \[
\Delta\le \max\left(\frac{8}{25}k^2 e^{0.6494-2 n / k-\left(\frac{1-D/k}{2-D/k}\right)\frac{D}{k}+2.051 / k},Dk\right)\times\left(1-\frac{u}{k}+\frac{u}{k} e^{-2 / k}\right).
\]
Further,
$$
1-\frac{u}{k}+\frac{u}{k} e^{-2 / k} \leq 1-\frac{2 u}{k^2}+\frac{2 u}{k^3} \leq e^{-2 u / k^2+2 u / k^3}.
$$
Hence,\begin{equation*}
\begin{aligned}
\Delta \leq & \max\left(\frac{8}{25} k^2 \exp \left(0.6494-\frac{2(n k+u-k)}{k^2}-\left(\frac{1-D / k}{2-D / k}\right) \frac{D}{k}+\frac{2.051}{ k}+\frac{2 u}{k^3}\right),\right.\\&\ \ \quad\ \ \left.Dk\exp\left(-\frac{2u}{k^2}+\frac{2u}{k^3}\right)\right) \\
= & \max\left(\frac{8}{25} k^2  \exp \left(0.6494-\frac{2(s-k)}{k^2}-\left(\frac{1-D / k}{2-D / k}\right) \frac{D}{k}+\frac{2.051}{ k}+\frac{2(s \bmod k)}{k^3}\right),\right.\\&\ \ \quad\ \  \left.Dk\exp\left(\frac{2(s \bmod k)}{k^2}\left(-1+\frac{1}{k}\right)\right)\right).
\end{aligned}
\end{equation*}
This completes the first part of Theorem \ref{newboundsvino}.\\
    We now turn to the second part of the theorem.
For $k< 90000$, we follow Ford's argument for proving Theorem \ref{th3ford} in \cite{ford_vinogradovs_2002}. Running Program 1 listed in Section \ref{sect9} with $\phi_J$ defined as in Lemma \ref{lemma3.4p}, we get the desired bounds for $k$ in the ranges
\[[129,137],\ [138,139],\ [140,146],\ [147,148],\ [149,170],\ [171,190],\ [191,339],\ [340,499].
\]For $500\le k< 90000$ we use Program 2 in Section \ref{sect9}, with $\phi_J$ defined as in Lemma \ref{lemma3.4p}. Program 2 differs from Program 1 in the definition of the variable $\operatorname{del0}$, as, due to \eqref{del500}, we know that, for $k\ge 500$, $\Delta_n\le 0.4k^2$ for every $n\ge n_0=\lceil0.1247k\rceil$. Hence, in order to initialize $\operatorname{del0}$, we start from $\Delta_{n_0}\le 0.4k^2$ instead of starting from $\Delta_1\le \frac{1}{2}k^2(1-\frac{1}{k})$.\\
For $k\ge 90000$, we follow Preobrazhenski\u{\i}'s argument \cite{preobrazhenskii_new_2011}. Taking $D=0.001k$ in Lemma \ref{lemma36new}, we observe that for every $k\ge 90000$, 
\[
2k\le n\le \frac{1}{2}k\left(0.6494+\log\left(\frac{8k}{25D}\right)-\left(\frac{1-D/k}{2-D/k}\right)\frac{D}{k}+\frac{2.051}{k}\right).
\]Hence, since from now on we will work with $D=0.001k$, we will restrict the range of $n$ in Lemma \ref{lemma36new} to
\begin{equation}\label{newrangen}
     2 k\le n\le \frac{1}{2}k\left(0.6494+\log\left(\frac{8k}{25D}\right)-\left(\frac{1-D/k}{2-D/k}\right)\frac{D}{k}+\frac{2.051}{k}\right)+1.
\end{equation}
Choosing
\begin{equation}\label{choicen}
n=\left\lceil\frac{1}{2}k\left(0.6494+\log\left(\frac{8k}{25D}\right)-\left(\frac{1-D/k}{2-D/k}\right)\frac{D}{k}+\frac{2.051}{k}\right)\right\rceil,
\end{equation}
we have, by Lemma \ref{lemma36new}, that
\begin{align*}
    \Delta_n&\le  \max\left(\frac{8}{25}k^2 e^{0.6494-2 n / k-\left(\frac{1-0.001}{2-0.001}\right)0.001+2.051 / k},0.001k^2\right)\\&\le \max\left(\frac{8}{25}k^2 e^{-\log\left(\frac{8}{25\cdot0.001}\right)},0.001k^2\right)\\&=0.001k^2,
\end{align*}
where we used the fact that
\[
-\frac{2n}{k}\le -\left(0.6494+\log\left(\frac{8k}{25D}\right)-\left(\frac{1-D/k}{2-D/k}\right)\frac{D}{k}+\frac{2.051}{k}\right).
\]
Hence, for every $k\ge 90000$ fixed, there exists $s\le \rho k^2$ such that $$
J_{s, k}(P) \leq k^{\theta k^3} P^{2 s-\frac{1}{2} k(k+1)+0.001 k^2},
$$where 
\[
\rho\le \frac{nk}{k^2}\le \frac{1}{2}\left(0.6494+\log\left(\frac{8k}{25D}\right)-\left(\frac{1-D/k}{2-D/k}\right)\frac{D}{k}+\frac{2.051}{k}\right)+\frac{1}{k}\le 3.20863
\]
and $\theta=2.17720$ ( the value of $\theta$ follows from a uniform upper bound for the constant
\[ k^{2.055 k^3-0.414 k^2+3 n k} 1.06^{n k^2+2\left(n^2-n\right) k+0.099912 k^3}
\]found in the first part of Theorem \ref{newboundsvino} which holds for every $k\ge 90000$, and $n$ in the range \eqref{newrangen}). This estimate holds since, using our choice \eqref{choicen} for $n$, we have $s=nk\le \rho k^2$ and $\Delta_n\le 0.001k^2.$
\begin{remark}
    The starting point $\Delta_{n_0}\le 0.4k^2$ with $n_0=\lceil0.1247k\rceil$ is nearly the optimal choice. Indeed, taking $\Delta_{n_0}\le ck^2$ with $c<0.4$ would imply a greater value for $n_0$. A direct computation shows that, although the constant $\frac{8}{25}$ in the estimate of $\Delta_n$ would decrease, the quantity $e^{0.6494}$ would increase, leading to an overall estimate for $\Delta_n$ which is worse than that one found by Ford in Lemma 3.6 of \cite{ford_vinogradovs_2002} when $\Delta_n>k$.\\ Furthermore, the lower bound for $k$ in \eqref{bound500} is optimal, as the inequality in \eqref{bound500} is no longer valid for $k<500$. Finally, also the value $0.101k^2$ in \eqref{bound500} is nearly optimal, as this inequality is satisfied only with $2s-ak^2$ and $a>0.1$.
\end{remark}
\section{Proof of Theorem \ref{theorem1}}\label{sect4}
 We will consider the cases $\lambda\le 84$ and $\lambda\ge 84$ separately. 
 \subsection{Case $\lambda\ge 84$}
We combine Ford's method \cite{ford_vinogradovs_2002} with the new estimates for the Vinogradov's integral found in Theorem \ref{newboundsvino} to obtain an improved upper bound for $S(N,t)$ in Lemma \ref{lemma5.1ford}.\\As in \cite{ford_vinogradovs_2002}, we make the following assumptions:\begin{equation*}
\left\lfloor M_1\right\rfloor \geq M_2 \geq 100 g, \quad s \leq 2^g, \quad r \geq 13 g, \quad r \geq s, \quad g \geq h \geq 3.
\end{equation*}
Also, we define
\begin{equation*}
M_1=N^{\mu_1}, \quad M_2=N^{\mu_2}, \quad \mu_1>\mu_2 
\end{equation*}
and\begin{equation}\label{condphigam}
\phi=g / \lambda, \quad \gamma=h / \lambda, \quad 1 \leq \gamma \leq \frac{1}{1-\mu_2}<\frac{1}{1-\mu_1} \leq \phi \leq \frac{1}{1-\mu_1-\mu_2}.
\end{equation}
Following the proof of Theorem 2 in \cite{ford_vinogradovs_2002} for $\lambda$ large, using the notation of Lemma \ref{lemma5.1ford}, we have 
\begin{equation}\label{w}
W_h \cdots W_g \leq 2^{g^2} M_2^{h+(h+1)+\cdots+g} N^{-H}
\end{equation}
where
\begin{equation}
\begin{aligned}
H \geq & \lambda^2\left(\phi+\gamma-\frac{\gamma^2}{2}-\frac{1-\mu_1-\mu_2}{2} \phi^2-\frac{2-\mu_1-\mu_2}{2\left(1-\mu_1\right)\left(1-\mu_2\right)}\right) \\
& \quad+\lambda\left(\frac{\gamma}{2}-\frac{\phi}{2}\left(1-\mu_1-\mu_2\right)\right)-\frac{2-\mu_1-\mu_2}{8} \\
&=: H_2 \lambda^2+H_1 \lambda-H_0.
\end{aligned}
\end{equation}
At this point, we shall take the near-optimal values\begin{equation}\label{parametres}
\mu_1  =0.1905, \quad \mu_2=0.1603, \quad k=\left\lfloor\frac{\lambda}{1-\mu_1-\mu_2}+0.000003\right\rfloor \geq 129,\quad
r  =\left\lfloor\rho k^2+1\right\rfloor,
\end{equation}where $\rho $ is taken from Theorem \ref{newboundsvino}.\\Furthermore, we choose $Y=288$ and we want to estimate $S(N,t)$ when $N\ge e^{Y\lambda^2}$, since otherwise for $N\le e^{Y\lambda^2}$ Theorem \ref{theorem1} follows trivially: \[
    S(N,t)\le N\le e^{Y/132.94357}N^{1-1/132.94357\lambda^2}\le 8.7979 N^{1-1/132.94357\lambda^2}.
    \]Finally, we consider\begin{equation}\label{g}
105 \leq g \leq 1.254 \lambda
\end{equation}and the following bounds for the quantity $k/\lambda$:\begin{equation*}
k_0:=\frac{1}{0.6492}-\frac{0.999997}{\lambda} \leq \frac{k}{\lambda} \leq \frac{1}{0.6492}+\frac{0.000003}{\lambda}=: k_1 \text {. }
\end{equation*}
First of all, by \eqref{parametres} and Theorem \ref{newboundsvino} we have \begin{equation}\label{vin}
\left\lfloor M_1\right\rfloor^{-2 r+\frac{1}{2} k(k+1)} J_{r, k}\left(\left\lfloor M_1\right\rfloor\right) \leq C_1 M_1^{0.001 k^2},
\end{equation}
where $C_1=k^{\theta k^3}$ and $\theta$ comes from Theorem \ref{newboundsvino}.\\Since $N\ge e^{Y\lambda^2}$, from \eqref{parametres} and \eqref{g} we have $M_2 \geq e^{\mu_2 Y \lambda^2} \geq e^{0.1019 Y g^2}$. Let $D=0.1019 Y=29.3472$ and $\eta=\frac{1}{\xi g^{3 / 2}}$, where $3 \leq \xi \leq 6$. By the assumption on $g$, \eqref{g}, the hypotheses of Theorem \ref{th4ford} are satisfied if we take $P=M_2$ and $k=g$. Hence, by Theorem \ref{th4ford},
\begin{equation}\label{inc}
   J_{s, g, h}\left(\mathscr{C}\left(M_2, M_2^\eta\right)\right) \leq C_2 P^{2 s-\frac{t}{2}(h+g)+E_2}, 
\end{equation}
where\begin{equation*}
\begin{aligned}
E_2 & =\frac{1}{2} t(t-1)+\frac{\eta s^2}{2 t}+h t \exp \left\{-\frac{s}{h t}\right\} \\
\log C_2 & =\frac{s^2}{t}+\frac{10.5 \xi^2 t g^2 \log ^2 g}{D}-s\left(\left(\xi g^{3 / 2}+h\right)(1-1 / h)^{s / t}-h\right) \log \left(\xi g^{3 / 2} / 10\right).
\end{aligned}
\end{equation*}Since by the hypotheses of Theorem \ref{th4ford} and the range \eqref{g} for $g$ we have$$
R=M_2^\eta \geq e^{D g^2 \eta} \geq g^{10}>6^{25},
$$following exactly the argument used by Ford in the proof of Theorem 2 in \cite{ford_vinogradovs_2002} for $\lambda$ large with $\delta=\frac{1}{25}$ instead of $\delta=\frac{1}{26}$ we get$$
\begin{aligned}
\frac{M_2}{\left|\mathscr{C}\left(M_2, R\right)\right|} & \leq(\log R)\left(1.0417 \xi g^{3 / 2}+1\right)\left(\frac{26.0417 \xi g^{3 / 2}}{2.5}\right)^{1.0417 \xi g^{3 / 2}} \\
& \leq(\log N) C_3 \leq C_3 N^{E_3},
\end{aligned}
$$where\begin{equation*}
\begin{aligned}
C_3 & =\left(10.4167 \xi g^{3 / 2}\right)^{1.0417 \xi g^{3 / 2}} \\
E_3 & =\frac{\log \left(Y \lambda^2\right)}{Y \lambda^2}.
\end{aligned}
\end{equation*}
By \eqref{parametres}, we have \begin{equation}\label{r}
(5 r)^k \leq\left(38.2 \lambda^2\right)^{1.55 \lambda} \leq \lambda^{4.65 \lambda}
\end{equation}
and
\begin{equation*}
    r\ge 7.6\lambda^2\ge \lambda^2.
    \end{equation*}
    It follows that
\begin{equation}\label{e3}
    \frac{E_3}{r} \leq \frac{\log \left(Y \lambda^2\right)}{7.6 Y \lambda^4}\le\frac{\log \left(Y \lambda^2\right)}{ Y \lambda^4} .
\end{equation}
    \begin{lemma}\label{medium}
    For $84\le\lambda\le 220$ and $N\ge e^{288\lambda^2}$ we have the following estimate
    \[
    S(N,t)\le 8.7979N^{1-1/132.94357\lambda^2}.
    \]
\end{lemma}
\begin{proof}
From Lemma \ref{lemma5.1ford}, \eqref{w}, \eqref{parametres}, \eqref{vin}, \eqref{inc} and \eqref{r} we have 
\begin{equation}\label{ccondw}
\begin{aligned}
S(N, t) & \leq\left(C_3^{\frac{1}{r}}\left(\lambda^{4.65 \lambda} C_1 C_2\right)^{\frac{1}{2 r s}}\right) N^{1+E}+2 N^{0.36}+\frac{1}{k} N^{1-0.0000019476} \\
E & =\frac{\log \left(Y \lambda^2\right)}{ Y \lambda^4}+\frac{1}{2 r s}\left(-H+0.001 \mu_1 k^2+\mu_2 E_2\right).
\end{aligned}
\end{equation}
Now, we choose 
\[\xi=3.612381,\qquad \sigma=0.330201.
\]
Furthermore, taking $s=\lfloor\sigma h t\rfloor+1$, we make the same assumptions as in Lemma 5.3 of \cite{ford_vinogradovs_2002}\begin{equation*}
g=\left\lfloor\frac{\lambda}{1-\mu_1}\right\rfloor+1+a, \quad h=\left\lfloor\frac{\lambda}{1-\mu_2}\right\rfloor-b, \quad t=g-h+1, \quad a, b \in\{0,1\},
\end{equation*}where $g$ satisfies \eqref{g}. Then, we bound the exponent of $N$ in each interval $\lambda \in I=\left[\lambda_1, \lambda_2\right)$ where each of the quantities $m_1=\left\lfloor\frac{\lambda}{1-\mu_1}\right\rfloor, m_2=\left\lfloor\frac{\lambda}{1-\mu_2}\right\rfloor$ and $k$  defined in \eqref{parametres} is constant. We also take constant values of $a$ and $b$ in $I$, so that $g, h, t, s, r$ are also fixed.\\As in \cite{ford_vinogradovs_2002}, we define for $\lambda\in I$ the quantities
\begin{equation*}
\begin{aligned}
H & =Z_0+Z_1 \lambda, \\
Z_0 & =\frac{\left(m_1^2+m_1\right)\left(1-\mu_1\right)+\left(m_2^2+m_2\right)\left(1-\mu_2\right)-h^2+h-\left(1-\mu_1-\mu_2\right)\left(g^2+g\right)}{2}, \\
Z_1 & =h+g-m_1-m_2-1=a-b \in\{-1,0,1\},
\end{aligned}
\end{equation*}so that$$
H \geq H^{\prime}:=Z_0+\left\{\begin{array}{ll}
\lambda_1 & Z_1=1 \\
0 & Z_1=0 \\
-\lambda_2 & Z_1=-1
\end{array}\right..
$$It follows that\begin{equation*}
E \leq \frac{\log \left(Y \lambda_1^2\right)}{Y \lambda_1^4}-\frac{H^{\prime}-0.001 \mu_1 k^2-\mu_2\left(\frac{t(t-1)}{2}+\frac{s^2}{\xi t g^{3 / 2}}+h t e^{-s /(h t)}\right)}{2 r s}:=E^{\prime} .
\end{equation*}
At this point we use Program 3 in Section \ref{sect9} to find the best values for $C$ and $u$, under the condition $u\le 132.94357$, so that, for $\lambda\in I$, $$
S(N, t) \leq C N^{1-1 /\left(u \lambda^2\right)}+\frac{1}{k} N^{1-1 /132 \lambda^2}.
$$ where $u=1 /\left(E^{\prime} \lambda_1^2\right)$ and $C=C_3^{1 / r}\left(\lambda^{4.65 \lambda} C_1 C_2\right)^{1 /(2 r s)}$. Running Program 3 we can notice that in each interval $I$ we have  $C\le 8.7979 $ and $u\le 132.94357$. The conclusion follows.
\end{proof}
\begin{lemma}\label{large}
    For $\lambda\ge 220$  and $N\ge e^{288\lambda^2}$ we have the following estimate
    \[
    S(N,t)\le 7.5N^{1-1/132.94357\lambda^2}.
    \]
\end{lemma}
\begin{proof}
From Lemma \ref{lemma5.1ford}, \eqref{w}, \eqref{parametres}, \eqref{vin}, \eqref{inc} and \eqref{r} we have 
\begin{equation}\label{ccondw}
\begin{aligned}
S(N, t) & \leq\left(C_3^{\frac{1}{r}}\left(\lambda^{4.65 \lambda} C_1 C_2\right)^{\frac{1}{2 r s}}\right) N^{1+E}+2 N^{0.36}+\frac{1}{k} N^{1-0.0000019476} \\
E & =\frac{\log \left(Y \lambda^2\right)}{ 7.6Y \lambda^4}+\frac{1}{2 r s}\left(-H+0.001 \mu_1 k^2+\mu_2 E_2\right).
\end{aligned}
\end{equation}
Now, we assume
\begin{equation}\label{condlemma5.2}
h=\left\lfloor 1.17928 \lambda+\frac{1}{2}\right\rfloor, \quad g=\left\lfloor 1.24788 \lambda+\frac{1}{2}\right\rfloor, \quad 
s=\lfloor\sigma h(t-1)+1\rfloor,
\end{equation}where
\begin{equation}\label{sigmalarge}
    \sigma=0.3299,\quad t=g-h+1.
\end{equation}
From \eqref{condphigam}, \eqref{condlemma5.2} and \eqref{sigmalarge}, the relation \eqref{g} holds and, furthermore, we have\begin{equation*}
|\gamma-1.17928| \leq \frac{1}{2\lambda}, \quad|\phi-1.24788| \leq \frac{1}{2\lambda}.
\end{equation*}From \eqref{parametres}, \eqref{condlemma5.2} and \eqref{sigmalarge} we have\begin{equation*}
g \geq 275, \quad h \geq 259, \quad t \geq 17\ge 13, \quad k \geq 338, \quad s \geq 0.02668 \lambda^2\ge 0.02294\lambda^2 .
\end{equation*}
Following exactly the proof of Lemma 5.2 in \cite{ford_vinogradovs_2002}, with $Y=288$ instead of $Y=300$, with the adjusted value of $D=0.1090Y=29.3472$ instead of $D=30.57$, we still have\begin{equation*}
C_3^{1 / r}\left(\lambda^{4.65 \lambda} C_1 C_2\right)^{1 / 2 r s}\le 7.5. 
\end{equation*}It remains to estimate $E$. We have
\begin{equation*}
\begin{aligned}
E \leq & \frac{\log \left(Y \lambda^2\right)}{7.6 Y \lambda^4}+\frac{-H+0.001 \mu_1 k^2}{2.00002 \rho \sigma \gamma(\phi-\gamma) \lambda^2 k^2}+\frac{\mu_2 E_2}{2 \rho k^2 s} \\
\leq & \frac{1.52 \times 10^{-7}}{\lambda^2}+\frac{-\lambda^2 H_2-\lambda H_1+H_0}{2.00002 \rho \sigma \gamma(\phi-\gamma) \lambda^2 k^2}+\frac{0.001 \mu_1}{2.00002 \rho \sigma \gamma(\phi-\gamma) \lambda^2} \\
& +\frac{\mu_2}{2 \rho k^2}\left[\frac{\phi-\gamma+1 / \lambda}{2 \sigma \gamma}+\frac{(t /(t-1)) e^{-\sigma+\sigma / t}}{\sigma}+\frac{\sigma h g^{-3 / 2}}{12}\right] .
\end{aligned}
\end{equation*}Following Ford's argument in \cite{ford_vinogradovs_2002}, we have\begin{equation*}
\lambda^2 E \leq 1.52 \times 10^{-7}+\frac{f(\gamma, \phi)+G_1 / \lambda^{1 / 2}+G_2 / \lambda}{\rho}
\end{equation*}where \begin{equation*}
\begin{aligned}
f(\gamma, \phi) & =\frac{1}{2.00002 \sigma \gamma}\left[\frac{0.001 \mu_1}{\phi-\gamma}+\frac{1}{k_1^2}\left(\frac{-H_2}{\phi-\gamma}+1.00001 \mu_2\left(\frac{1}{2}(\phi-\gamma)+\gamma e^{-\sigma}\right)\right)\right], \\
G_1 & =\frac{\mu_2 \sigma \gamma \phi^{-3 / 2}}{24 k_0^2} \leq 0.0008, \\
G_2 & =\frac{1}{2.00002 \sigma(k / \lambda)^2}\left[\frac{-H_1+H_0 / \lambda+1.33547 \mu_2 \gamma e^{-\sigma}}{\gamma(\phi-\gamma)}+\frac{1.00001 \mu_2}{2 \gamma}\right] .
\end{aligned}
\end{equation*}
Using Ford's estimate for the expression inside the brackets in the definition of $G_2$, we have \begin{equation*}
G_2 \le \frac{0.0392}{2.00002 \sigma \gamma k_0^2} \le 0.0213552.
\end{equation*}It follows that
\begin{equation*}
\begin{aligned}
\lambda^2 E & \le 1.56 \times 10^{-7}+\frac{f(\gamma, \phi)+0.0008 \lambda^{-1 / 2}+0.0213552 \lambda^{-1}}{\rho} \\
& \le 0.0000473+\frac{f(\gamma, \phi)}{\rho} .
\end{aligned}
\end{equation*}In the range $$|\phi-1.24788| \leq \frac{1}{440},\qquad|\gamma-1.17928| \leq \frac{1}{440},$$ $f(\gamma,\phi)$ is increasing, hence the  maximum occurring at $\gamma=1.17928+\frac{1}{440}, \phi=1.24788-\frac{1}{440}$, with
$$
f(\gamma, \phi) \leq-0.024287046496.
$$
Using $\rho=3.20863$ from Theorem \ref{newboundsvino}, it follows that
$$
\lambda^2 E \le-\frac{1}{132.94357}.
$$
\end{proof}
Theorem \ref{theorem1} for $\lambda\ge 84$ follows directly from Lemma \ref{medium} and Lemma \ref{large}.
\begin{remark}
    One might have considered the ranges $84\le \lambda\le N$ and $\lambda> N$, with $N>220$. However, a direct computation shows that even for really large values of $N$, of the order of $10^6$, and hence, by definition of $k$, for values of $k$ much greater than $90000$, the improvement is negligible. Indeed, one can see that a new choice for $N$ would not influence the estimate for $S(N,t)$ in the range $84\le \lambda\le N$ we found in Lemma \ref{medium}, while for the case $\lambda>N$, if $N$, and hence $k$, would be sufficiently large, the value of $\rho$ in Theorem \ref{newboundsvino} would be $3.20861$ for $k$ sufficiently large, which is slightly smaller that $\rho=3.20863$ we found for $k\ge 90000$. However, a direct computation shows that this really small improvment on $\rho$ for $k$ sufficiently large would not lead to an improvement in Lemma \ref{large}, when $\lambda>N$ and $N$ sufficently large.
\end{remark}
 \subsection{Case $\lambda\le 84$} 
For $\lambda\le 84$, the values for the $C$ constant found in \cite{ford_vinogradovs_2002} hold, both for $1\le \lambda\le2.6$ and $2.6\le \lambda\le 84$. Indeed, for the range $2.6\le \lambda\le 84$, the same values for $C$ still hold under the new constraint that $S(N,t)\le CN^{1-1/u\lambda^2}$ with $u\le 132.94357$.  We can notice that the maximum is reached inside the interval $\lambda\in[83,84]$, where $C=8.7979$. This concludes the proof of Theorem \ref{theorem1} also for $\lambda\le 84$.
\section{Proof of Theorem \ref{theorem2}}
 First of all, since $\zeta(\bar{s}, u)=\overline{\zeta(s, u)}$ and $\zeta(s)=\zeta(s,1)$, we restrict our attention to $s$ lying in the upper half-plane. Then, we consider separately the cases $\sigma\leq\frac{15}{16}$ or $t\leq 10^{108}$ or $t\geq 10^{108}$ and $\frac{15}{16}\geq \sigma\geq 1$. The main contribution will come from the case $t\geq 10^{108}$ and $\frac{15}{16}\geq \sigma\geq 1$.
\subsection{Case $\frac{15}{16}\leq\sigma\leq 1,\ t\geq 10^{108}$}Using Lemma \ref{lemma73} with the values $C=8.7979$ and $D=132.94357$ found in Theorem \ref{theorem1} and $t\ge 10^{108}$, we have\[\left(\frac{C+1+10^{-80}}{\log ^{2 / 3} t}+1.569 C D^{1 / 3}\right)\le \left(\frac{C+1+10^{-80}}{\log ^{2 / 3} (10^{108})}+1.569 C D^{1 / 3}\right)\le 70.6995. \]
It follows that \[
    |\zeta(\sigma+it)|\leq 70.6995 t^{4.43795(1-\sigma)^{3/2}}\log^{2/3}t
    \]
    and
    \[
     \left|\zeta(s, u)-u^{-s}\right| \leq70.6995 t^{4.43795(1-\sigma)^{3/2}}\log^{2/3}t.
    \]
    \subsection{Case $\frac{15}{16}\leq\sigma\leq 1,\ 3\leq t\leq 10^{108}$}Following the proof of Lemma \ref{lemma7.1} in \cite{ford_vinogradovs_2002}, one has\[\left|\zeta(s, u)-u^{-s}\right| \leq(t+3 / 2)^{1-\sigma}(1+1 / t+\log (2 t+1))\]for $\frac{15}{16}\le \sigma\le 1$ and $3\le t\le 10^{108}$.\\If $3\le t\le 10^6$, from \cite{ford_vinogradovs_2002} we know that \[\left|\zeta(s, u)-u^{-s}\right| \leq36.8.\] If $10^6\le t\le 10^{108}$, from the argument of Lemma \ref{lemma7.1} in \cite{ford_vinogradovs_2002}, one gets
    \begin{equation*}
        \begin{aligned}
           \left|\zeta(s, u)-u^{-s}\right|&\leq(t+3 / 2)^{1-\sigma}(1+1 / t+\log (2 t+1)) \\&\leq 1.123 t^{1-\sigma} \log t\\&=1.123\left(t^{4(1-\sigma)^{3 / 2}} \log ^{2 / 3} t\right)\left(t^{1-\sigma-4(1-\sigma)^{3 / 2}} \log ^{1 / 3} t\right)\\&\le 1.123\left(t^{4(1-\sigma)^{3 / 2}} \log ^{2 / 3} t\right)\left(t^{\frac{1}{108}} \log ^{1 / 3} t\right)\\&\leq 70.6199 t^{4(1-\sigma)^{3/2}}\log^{2/3}t\\&\leq70.6995 t^{4.43795(1-\sigma)^{3/2}}\log^{2/3}t.
        \end{aligned}
    \end{equation*}Hence, we also have
    \[
|\zeta(\sigma+it)|\le 70.6199 t^{4(1-\sigma)^{3/2}}\log^{2/3}t\leq70.6995 t^{4.43795(1-\sigma)^{3/2}}\log^{2/3}t.
    \]
   \begin{remark}
     One can notice that the estimate found for $\zeta(s)$ and $\zeta(s,u)$ in this range is much better than that one found for $\frac{15}{16}\le \sigma\le 1$ and $t\ge 10^{108}$. Indeed, the $B$ constant we found in this current case is just $4$, which is less than the final value $4.43795$.
   \end{remark} 
    \subsection{Case $\frac{1}{2}\leq\sigma\leq \frac{15}{16},\ t\geq 3$} From the proof of Lemma \ref{lemma7.1} in \cite{ford_vinogradovs_2002} we have \begin{equation}\label{smallestbound}
        |\zeta(\sigma +it)|\leq 21.3 t^{4(1-\sigma)^{3/2}}\leq70.6995 t^{4.43795(1-\sigma)^{3/2}}\log^{2/3}t.
    \end{equation}
\begin{remark}
The choice of $\frac{15}{16}$ comes from the proof of Lemma 7.1 in \cite{ford_vinogradovs_2002}, where $(1-\sigma)\leq 4(1-\sigma)^{3/2}$ for $\sigma\leq \frac{15}{16}$.\\With $(1-\sigma)\leq 4.4(1-\sigma)^{3/2}$ one should get $\sigma\leq \frac{459}{484}$. However, the influence of this second choice on the $A$ constant is negligible and it does not lead to any further improvements on $A$.
\end{remark}
\section{Proof of Theorem \ref{theorem6}}
For $t\le \exp(463388)$ we use the best known explicit Littlewood zero-free region due to Yang in \cite{yang2023explicit}. For $t\ge \exp(463388)$ we follow the argument used to prove Theorem 1.1 in \cite{mossinghoff_explicit_2022} which relies on a non-negative trigonometric polynomial $P_{40}(x)$ defined by \eqref{deftrpol} with degree 40 having
$$
b_0=1, \quad b_1=1.74600190914994, \quad b=\sum_{k=1}^{40} b_k=3.56453965437134.
$$Using  the new values $A=70.6995$ and $B=4.43795$ found in Theorem \ref{theorem2}  and making the following new assumptions
\[
T_0:=\exp(463388),\qquad M_1:=0.050007,\qquad E=1.8008278,\qquad R=468
\]in \cite{mossinghoff_explicit_2022},
the proof of Theorem \ref{theorem6} is complete.

\section{Proof of Theorem \ref{theoremasimp}}
The proof is exactly the same as in \cite{mossinghoff_explicit_2022}, except for the new value $B=4.43795$ found in Theorem \ref{theorem2}, and relies on a non-negative polynomial $P_{46}(x)$ with degree 46 where
$$
b_0=1, \quad b_1=1.74708744081848, \quad b=\sum_{k=1}^{46} b_k=3.57440943022073.
$$
\section{Some possible improvements on $A$}
In this section we give some suggestions for some possible improvements on the constant $A$ in Theorem \ref{theorem2}. We will provide some quite detailed proofs of some useful lemmas to help the reader follow the argument more easily.\\

Instead of the definition \eqref{Snt} for $S(N,t)$, one could consider the following sum\begin{equation}\label{ms}
    \Tilde{S}(N, t) := \max_{0 < u \le 1}\max_{N < R \le mN}\left|\sum_{N \le n \le R}\frac{1}{(n + u)^{it}}\right|
\end{equation}
with $1<m\leq 2$. For $\lambda\ge 84$, we can find sharper estimates.
\begin{theorem}\label{theorem3}
    For $1.001\le m\le 2$ and $\lambda\ge 84$ we have:
    \[
    \Tilde{S}(N,t)\le (m-1)8.7979N^{1-1/132.94357\lambda^2}.
    \]
\end{theorem}
If one could find a bound of the form $\Tilde{S}(N,t)\le (m-1)cN^{1-1/(u\lambda^2)}$ also for $\lambda\le 84$, then the constant $A$ in Theorem \ref{theorem2} might be reduced to around $49$. We briefly outline the proof of Theorem \ref{theorem3}.\\
We start with a preliminary lemma that is a more general version of Lemma \ref{lemma5.1ford}. 
\begin{lemma}\label{lemma1}Suppose $k, r$ and $s$ are integers $\geq 2$, and $h$ and $g$ are integers satisfying $1 \leq h \leq g \leq k$. Let $N$ be a positive integer, and $M_1, M_2$ be real numbers with $1 \leq M_i \leq N$. Let $\mathscr{B}$ be a nonempty subset of the positive integers $\leq M_2$. Then\begin{align*}
    &\Tilde{S}(N,t)\leq 2M_1 M_2+\frac{(m-1)t(M_1 M_2)^{k+1}}{(k+1)N^k}\\&+(m-1)N\left(\frac{M_2}{|\mathscr{B}|}\right)^{1/r}\left((5r)^kM_2^{-2s}\lfloor M_1\rfloor ^{-2r+k(k+1)/2}J_{r,k}(\lfloor M_1\rfloor)J_{s,g,h}(\mathscr{B})W_h\dots W_g\right)^{1/2rs},
\end{align*}
    where
$$
W_j=\min \left(2 s M_2^j, \frac{2 s M_2^j}{r\left\lfloor M_1\right\rfloor^j}+\frac{s t M_2^j}{\pi j N^j}+\frac{4 \pi j(2 N)^j}{r t\left\lfloor M_1\right\rfloor^j}+2\right) \quad(j \geq 1) 
$$and $\Tilde{S}(N,t)$ is defined in \eqref{ms}.
\end{lemma}
\begin{proof}
We define $M=\left\lfloor M_1\right\rfloor$. For $N<R \leq m N$, $1< m\le 2$ and $0<u \leq 1$, we have
\begin{equation*}
    \begin{aligned}
        \left|\sum_{N<n \leq R}(n+u)^{-i t}\right|&=\frac{1}{M|\mathscr{B}|}\left|\sum_{\substack{a \leq M_1 \\
b \in \mathscr{B}}} \sum_{N<n+a b \leq R}(n+a b+u)^{-i t}\right| \\&
\leq \frac{1}{M|\mathscr{B}|}\left|\sum_{\substack{a \leq M_1 \\
b \in \mathscr{B}}} \sum_{\substack{N<n \leq R-1}}(n+a b+u)^{-i t}\right|+\frac{1}{M|\mathscr{B}|} \sum_{\substack{a \leq M_1 \\
b \in \mathscr{B}}}(2 a b-1) \\&
\leq \frac{N}{M|\mathscr{B}|} \max _{N \leq z \leq m N}\left|\sum_{\substack{a \leq M_1 \\
b \in \mathscr{B}}} e^{-i t \log (1+a b / z)}\right|+2 M_1 M_2.
    \end{aligned}
\end{equation*}
For $0 \leq x \leq 1$ we have
$$
\left|\log (1+x)-\left(x-x^2 / 2+\cdots+(-1)^{k-1} x^k / k\right)\right| \leq \frac{x^{k+1}}{k+1}.
$$
Also $\left|e^{i y}-1\right| \leq y$ for real $y$ and $a b / z \leq M_1 M_2 / N$. Thus, for some $z \in[N, m N]$,
$$
\Tilde{S}(N, t) \leq \frac{(m-1)N}{M|\mathscr{B}|}|U|+\frac{t(m-1)\left(M_1 M_2\right)^{k+1}}{(k+1) N^k}+2 M_1 M_2
$$
where $U=\sum_{a, b} e\left(\gamma_1(a b)+\cdots+\gamma_k(a b)^k\right)$ and $\gamma_j=(-1)^j t /\left(2 \pi j z^j\right)$.\\At this point, from Ford's argument in Lemma 5.1 of \cite{ford_vinogradovs_2002}, we have
\[
\frac{N}{M|\mathscr{B}|}|U|\le N\left(\frac{M_2}{|\mathscr{B}|}\right)^{\frac{1}{r}}\left((5 r)^k M_2^{-2 s}\left\lfloor M_1\right\rfloor^{-2 r+\frac{1}{2} k(k+1)} J_{r, k}\left(\left\lfloor M_1\right\rfloor\right) J_{s, g, h}(\mathscr{B}) W_h \cdots W_g\right)^{\frac{1}{2 r s}}.
\]
The conclusion follows.
\end{proof}
At this point, from Ford's argument in section $5$ of \cite{ford_vinogradovs_2002} we have the following bounds:
\[
\frac{t(M_1 M_2)^{k+1}}{(k+1)N^k}\le \frac{1}{k} N^{1-0.0000019476},
\]
\[
2M_1M_2\le 2 N^{0.36}
\]
and 
\begin{equation*}
\begin{aligned}
    &N\left(\frac{M_2}{|\mathscr{B}|}\right)^{1/r}\left((5r)^kM_2^{-2s}\lfloor M_1\rfloor ^{-2r+k(k+1)/2}J_{r,k}(\lfloor M_1\rfloor)J_{s,g,h}(\mathscr{B})W_h\dots W_g\right)^{1/2rs}\\&\le \left(C_3^{\frac{1}{r}}\left(\lambda^{5 \lambda} C_1 C_2\right)^{\frac{1}{2 r s}}\right) N^{1+E}
    \end{aligned}
\end{equation*}
where$$
E=\frac{\log \left(Y \lambda^2\right)}{7.6 Y \lambda^4}+\frac{1}{2 r s}\left(-H+0.001 \mu_1 k^2+\mu_2 E_2\right)
$$
and $C_1,C_2,C_3$ are the same as for Theorem \ref{theorem1}. If we use these bounds in Lemma \ref{lemma1} we get
\begin{equation}\label{eqfin}
\Tilde{S}(N, t) \leq(m-1)\left(C_3^{\frac{1}{r}}\left(\lambda^{4.65 \lambda} C_1 C_2\right)^{\frac{1}{2 r s}}\right) N^{1+E}+2 N^{0.36}+\frac{(m-1)}{k} N^{1-0.0000019476}.
\end{equation}
Now, we consider $N\ge e^{288\lambda^2}$, since otherwise we have trivially
\[
\Tilde{S}(N,t)\le (m-1)N\le (m-1)e^{288/132.94357}N^{1-1/132.94357\lambda^2}\le (m-1)8.7979N^{1-1/132.94357\lambda^2} .
\]
\begin{lemma}\label{mediumm}
    For $84\le\lambda\le 220$ and $N\ge e^{288\lambda^2}$ we have the following estimate:
    \[
    \Tilde{S}(N,t)\le (m-1)8.797901N^{1-1/132.94357\lambda^2}.
    \]
\end{lemma}
\begin{proof}
From Lemma \ref{medium} we know that 
\[
\left(C_3^{\frac{1}{r}}\left(\lambda^{4.65 \lambda} C_1 C_2\right)^{\frac{1}{2 r s}}\right) N^{1+E}\le 8.7979N^{1-1/132.94357\lambda^2}.
\]
Using this result in \eqref{eqfin} we get
\begin{equation*}
\begin{aligned}
\Tilde{S}(N, t) &\leq(m-1)8.7979N^{1-1/132.94357\lambda^2}+2 N^{0.36}+\frac{(m-1)}{k} N^{1-0.0000019476}\\& =(m-1)\left(8.7979N^{1-1/(132.94357\lambda^2)}+\frac{2N^{0.36}}{m-1}+\frac{1}{k+1}N^{1-1/132\lambda^2}\right)\\&\leq (m-1)8.797901N^{1-1/(132.94357\lambda^2)}
\end{aligned}
\end{equation*}
where the last passage comes from the fact that, for $m=1.001$, we have \[
(m-1)\cdot 10^{-7}\cdot8.7979N^{1-1/(132.94357\lambda^2)}\geq 2N^{0.36}.
\]Indeed, for $84\le\lambda\le 220$ and $N\ge e^{288\lambda^2}$, we have
\begin{equation*}
\begin{aligned}
    N^{1-1/(132.94357\lambda^2)-0.36}&\ge e^{288\lambda^2(1-1/(132.94357\lambda^2)-0.36)}\\&\ge e^{288\cdot 84^2(1-1/(132.94357\cdot 84^2)-0.36)} \\&\ge \frac{2}{8.7979\cdot 10^{-7}\cdot(m-1)}
\end{aligned}
\end{equation*}

if
\[
m-1\ge \frac{2\cdot 10^7}{8.7979}\cdot e^{-288\cdot 84^2(1-1/(132.94357\cdot 84^2)-0.36)}\ge 10^{-560000},
\]
which holds, since with our choice $m-1=10^{-3}$.
\end{proof}
\begin{lemma}\label{largem}
 For $\lambda\ge 220$ and $N\ge e^{288\lambda^2}$ we have the following estimate:
    \[
    \Tilde{S}(N,t)\le (m-1)7.5001N^{1-1/132.94357\lambda^2}.
    \]
\end{lemma}
\begin{proof}
From Lemma \ref{large} we know that 
\[
\left(C_3^{\frac{1}{r}}\left(\lambda^{4.65 \lambda} C_1 C_2\right)^{\frac{1}{2 r s}}\right) N^{1+E}\le 7.5N^{1-1/132.94357\lambda^2}.
\]
Using this result in \eqref{eqfin} we get
\begin{equation*}
\begin{aligned}
\Tilde{S}(N, t) &\leq(m-1)7.5N^{1-1/132.94357\lambda^2}+2 N^{0.36}+\frac{(m-1)}{k} N^{1-0.0000019476}\\& \le(m-1)\left(7.5N^{1-1/(132.94357\lambda^2)}+\frac{2N^{0.36}}{m-1}+\frac{1}{k+1}N^{1-1/132.94357\lambda^2}\right)\\&\leq (m-1)7.5001N^{1-1/(132.94357\lambda^2)}
\end{aligned}
\end{equation*}
where the last passage comes from the fact that, for $m=1.001$, we have \[
(m-1)\cdot 10^{-5}\cdot7.5N^{1-1/(132.94357\lambda^2)}\geq 2N^{0.36}.
\]Indeed, for $\lambda\ge 220$ and $N\ge e^{288\lambda^2}$, we have
\begin{equation*}
    \begin{aligned}
        N^{1-1/(132.94357\lambda^2)-0.36}&\ge e^{288\lambda^2(1-1/(132.94357\lambda^2)-0.36)}\\&\ge e^{288\cdot 220^2(1-1/(132.94357\cdot 220^2)-0.36)}\\& \ge \frac{2}{7.5\cdot 10^{-5}(m-1)}
    \end{aligned}
\end{equation*}
if
\[
m-1\ge \frac{2\cdot 10^5}{7.5}\cdot e^{-288\cdot 220^2(1-1/(132.94357\cdot 220^2)-0.36)}\ge 10^{-3800000},
\]
which holds, since with our choice $m-1=10^{-3}$.
\end{proof}
Theorem \ref{theorem3} follows immediately from Lemma \ref{mediumm} and Lemma \ref{largem}.\\At this point, if one could find an estimate of the form $\Tilde{S}(N,t)\le (m-1)cN^{1-1/(u\lambda^2)}$ also for $\lambda\le 84$, then we can use a more general version of Lemma \ref{lemma73}.
\begin{lemma}\label{lemma7.2new}
If $s=\sigma+i t, \frac{15}{16} \leq \sigma \leq 1, t \geq 10^{90}$ and $0<u \leq 1$, then
$$
\left|\zeta(s, u)-\sum_{0 \leq n \leq t}(n+u)^{-s}\right| \leq 10^{-77}.
$$
\end{lemma}
\begin{proof}
    The proof is the same as that one of Lemma \ref{lemma7.2} in \cite{ford_vinogradovs_2002}, with $t\ge 10^{90}$ instead of $t\ge 10^{100}$.
\end{proof}
\begin{lemma}\label{finallemma}
    Suppose that $\Tilde{S}(N, t) \leq (m-1)C N^{1-1 /\left(D \lambda^2\right)}(1 \leq N \leq t)$ for positive constants $C$ and $D$ and $1<m\le 2$, where $\lambda=\frac{\log t}{\log N}$. Let $B=\frac{2}{9} \sqrt{3 D}$. Then, for $\frac{15}{16} \leq \sigma \leq 1, t \geq 10^{90}$ and $0<u \leq 1$, we have
$$
|\zeta(s)| \leq\left(\frac{(m-1)C+1+\frac{1}{10^{77}}}{\log ^{2 / 3}t}+\frac{1.0875034 (m-1)C D^{1/3}}{\log (m)}\right) t^{B(1-\sigma)^{3 / 2}} \log ^{2 / 3} t.
$$
\end{lemma}
\begin{proof}
 Let
$$
S_1(u)=\sum_{1 \leq n \leq t}(n+u)^{-s}
$$
By Lemma \ref{lemma7.2new}, $\left|\zeta(s, u)-u^{-s}\right| \leq 10^{-77}+S_1(u)$. Put $r=\left\lceil\frac{\log t}{\log m}\right\rceil$. By partial summation,
\begin{equation*}
\begin{aligned}
\left|S_1(u)\right| & \leq 1+\sum_{j=0}^{r-1}\left|\sum_{m^j<n \leq \min \left(t, m^{j+1}\right)}(n+u)^{-\sigma-i t} \right| \\
& \leq 1+\sum_{j=0}^{r-1}\left(m^j\right)^{-\sigma} S\left(m^j, t\right) \\
& \leq 1+(m-1)C \sum_{j=0}^{r-1} e^{g(j)},
\end{aligned}
\end{equation*}
where \begin{equation*}
g(j)=(1-\sigma)(j \log m)-\frac{(j \log m)^3}{D \log ^2 t}.
\end{equation*}
As a function of $x,\  g(x)$ is increasing on $\left[0, x_0\right]$ and decreasing on $\left[x_0, \infty\right)$, where $x_0 \log m=\sqrt{D(1-\sigma) / 3} \log t$. Thus
\begin{equation*}
\begin{aligned}
\frac{\left|S_1(u)\right|-1}{(m-1)C} & \leq e^{g\left(x_0\right)}+\int_0^r e^{g(x)} d x \\
& \leq t^{B(1-\sigma)^{3 / 2}}+\frac{D^{1 / 3} \log ^{2 / 3} t}{\log m} \int_{0}^{\infty} e^{3 y^2 u-u^3} d u,
\end{aligned}
\end{equation*}
where $y=\sqrt{(1-\sigma) / 3} D^{1 / 6} \log ^{1 / 3} t$.\\ To bound the last integral, we make use of the inequality
$$
e^{-2 y^3} \int_0^{\infty} e^{3 y^2 u-u^3} d u \leq 1.0875034 \quad(y \geq 0),
$$
where the maximum occurs near  $y=0.710$. Therefore
\begin{equation*}
\frac{\left|S_1(u)\right|-1}{(m-1)C} \leq t^{B(1-\sigma)^{3 / 2}}\left(1+\frac{1.0875034}{\log m} D^{1 / 3} \log ^{2 / 3} t\right)
\end{equation*}
\end{proof}
However, Ford's method in \cite{ford_vinogradovs_2002} for $\lambda\le 87$ (or $\lambda\le 84$ as in our paper), cannot be modified to extract the factor $(m-1)$. Indeed, following the  proof of Lemma 6.3 in \cite{ford_vinogradovs_2002}, one should estimates the following quantity $$
\left(\sum_{N<n \leq m N-1}|T(n)|^{2 s}\right)^{1 / 2 s},
$$
where
$$
T(n)=\sum_{v \leq M} e\left(-\frac{t}{2 \pi} \log \left(1+\frac{v}{n+u}\right)\right).
$$ However, since an estimate on the whole sum of $T(n)$ over $N\le n\le mN$ is required, instead of one for the maximum of $T(n)$ over the interval $N\le n\le mN$, it is not possible to extract a factor $(m-1)$ at this step.  A clever argument which would overcome this problem could lead to a suitable estimate also for the case $\lambda\le 84$, with a consequent improved value for $A$.
\section{Code listings}\label{sect9}
\par\noindent\rule[5mm]{\textwidth}{0.5pt}
\medskip
\textbf{Program 1 for Theorem \ref{newboundsvino}}
\begin{lstlisting}
#include <stdio.h>
#include <math.h>
#include <stdlib.h>
#define max(x,y) (((x)>(y))?(x):(y))
#define min(x,y) (((x)>(y))?(y):(x))

double newdel(double k, double r, double del)
{
    double y,p,pf,tkr;
    long j,jf,jj;
    if ((r<4.0) && (r>k)) return (2.0*del);
    tkr=2.0*k*r; y=2.0*del-(k-r)*(k-r+1.0);
    if ((y<0.0) && (2.0*k/(tkr+y))<= 1.0/(k+1.0)) 
       return (del*2.0);

    j=floor(min(0.5*(3.0+sqrt(4.0*y+1.0)), 9.0*r/10.0));
    p=1.0/r;
    for (jj=j-1; jj>=1; jj--) {
        p=0.5/r+0.5*(1.0-y/(tkr-2.0*r*jj))*p;
    }
    return(del-k+0.5*p*(tkr-y));
}

int main()
{
    long j,k,k0,k1,r,r0,r1,n,bestr,s;
    double kk,logk,del0,del1,bestdel,goal,maxs,eta,om;
    double logH,logW,logC,k3,theta,thetamax;
    printf("enter k range :"); scanf("%ld %ld", &k0,&k1);
    maxs=0.0; thetamax=0.0;
    for (k=k0; k<=k1; k++) {
        kk=(double)k;
        logk=log(kk); k3=kk*kk*kk*logk;
        om=0.5; 
        for (j=1;j<=10;j++) om=1.5/(log(18.0*k3/om)-1.5);
        eta=1.0+om;
        logW=(kk+1.0)*max(1.5+1.5/om,log(18.0/om*k3));
        del0=0.5*kk*kk*(1.0-1.0/kk);
        goal=0.001*kk*kk;
        logH=3.0*kk*logk+(kk*kk-4.0*kk)*log(eta);
        logC=kk*logk;
        for (n=1; ;n++) {
            r0=(long)(sqrt(kk*kk+kk-2.0*del0)+0.5)-2; 
            r1=r0+4;
            bestdel=kk*kk; bestr=-1;
            for (r=r0;r<=r1;r++) {
                del1=newdel(kk,(double)r,del0);
                if (del1<bestdel) { bestdel=del1; bestr=r;}
            }
            del1=bestdel; r=bestr;
            if ((del1>=del0) && (r<r0)) exit(-1);
            logC +=max(logH+4.0*kk*n*log(eta),logW*(del0-del1));
            if (del1<=goal) {
                s=(long)((n+(del0-goal)/(del0-del1))*kk+1);
                theta=logC/k3;
                printf("%4d: s=%8.6f k^2 eta=%9.7f theta=%10.8f\n",
                       k,s/kk/kk,eta,theta);
                if ((s/kk/kk)>maxs) maxs=s/kk/kk;
                if (theta>thetamax) thetamax=theta;
                break;
            }
            del0=del1;
        }
    }
    printf("\n max s=%9.6fk^2 maxtheta=%10.8f\n",maxs,thetamax);
    system("pause");
}

\end{lstlisting}
\medskip
\par\noindent\rule[5mm]{\textwidth}{0.5pt}
\medskip
\textbf{Program 2 for Theorem \ref{newboundsvino}}
\begin{lstlisting}
#include <stdio.h>
#include <math.h>
#include <stdlib.h>
#define max(x,y) (((x)>(y))?(x):(y))
#define min(x,y) (((x)>(y))?(y):(x))

double newdel(double k, double r, double del)
{
    double y,p,pf,tkr;
    long j,jf,jj;
    if ((r<4.0) && (r>k)) return (2.0*del);
    tkr=2.0*k*r; y=2.0*del-(k-r)*(k-r+1.0);
    if ((y<0.0) && (2.0*k/(tkr+y))<= 1.0/(k+1.0)) 
       return (del*2.0);

    j=floor(min(0.5*(3.0+sqrt(4.0*y+1.0)), 9.0*r/10.0));
    p=1.0/r;
    for (jj=j-1; jj>=1; jj--) {
        p=0.5/r+0.5*(1.0-y/(tkr-2.0*r*jj))*p;
    }
    return(del-k+0.5*p*(tkr-y));
}

int main()
{
    long j,k,k0,k1,r,r0,r1,n,bestr,s;
    double kk,logk,del0,del1,bestdel,goal,maxs,eta,om;
    double logH,logW,logC,k3,theta,thetamax;
    printf("enter k range :"); scanf("%ld %ld", &k0,&k1);
    maxs=0.0; thetamax=0.0;
    for (k=k0; k<=k1; k++) {
        kk=(double)k;
        logk=log(kk); k3=kk*kk*kk*logk;
        om=0.5; 
        for (j=1;j<=10;j++) om=1.5/(log(18.0*k3/om)-1.5);
        eta=1.0+om;
        logW=(kk+1.0)*max(1.5+1.5/om,log(18.0/om*k3));
        del0=0.4*kk*kk;
        goal=0.001*kk*kk;
        logH=3.0*kk*logk+(kk*kk-4.0*kk)*log(eta);
        logC=kk*logk;
        for (n=ceil(0.1247*kk); ;n++) {
            r0=(long)(sqrt(kk*kk+kk-2.0*del0)+0.5)-2; r1=r0+4;
            bestdel=kk*kk; bestr=-1;
            for (r=r0;r<=r1;r++) {
                del1=newdel(kk,(double)r,del0);
                if (del1<bestdel) {bestdel=del1; bestr=r; }
            }
            del1=bestdel; r=bestr;
            if ((del1>=del0) && (r<r0)) exit(-1);
            logC +=max(logH+4.0*kk*n*log(eta),logW*(del0-del1));
            if (del1<=goal) {
                s=(long)((n+(del0-goal)/(del0-del1))*kk+1);
                theta=logC/k3;
                printf("%4d: s=%8.6f k^2 eta=%9.7f theta=%10.8f\n",
                       k,s/kk/kk,eta,theta);
                if ((s/kk/kk)>maxs) maxs=s/kk/kk;
                if (theta>thetamax) thetamax=theta;
                break;
            }
            del0=del1;
        }
    }
    printf("\n max s=%9.6fk^2 maxtheta=%10.8f\n",maxs,thetamax);
    system("pause");
}

\end{lstlisting}
\medskip
\par\noindent\rule[5mm]{\textwidth}{0.5pt}
\medskip
\textbf{Program 3 for Lemma \ref{medium}}
\begin{lstlisting}
#include<stdio.h>
#include <math.h>
long k,g,h,s,r,t,g0,h0,g1,h1,flag;
double mu1,mu2,xi,lam,lam1,lam2,D,sigma,Y,goal;
void calc(ex,c,pr)
    double *ex,*c; int pr;
{
    double kk,logk,k2,log(),exp(),floor(),ceil();
    double th,rr,ss,tt,gg,hh,rho,H,E1,E2,E3,m1,m2,Z0,Z1,
           reta,logC1,logC2,logC3,logC,dc;
    k=(long) (lam/(1.0-mu1-mu2)+0.000003);
    kk=(double) k;
    logk=log(kk); k2=kk*kk;
    rho=3.20863; th=2.17720;
    if (k<=89999) {rho=3.205502; th=1.77775;}
    if (k<=499) {rho=3.196497; th=2.24352;}
    if (k<=339) {rho=3.192950; th=2.33313;}
    if (k<=190) {rho=3.184127; th=2.35334;}
    if (k<=170) {rho=3.181869; th=2.37929;}
    if (k<=148) {rho=3.178871; th=2.38259;}
    if (k<=146) {rho=3.178551; th=2.39167;}
    if (k<=139) {rho=3.177527; th=2.39529;}
    if (k<=137) {rho=3.177207; th=2.40930;}
    r=(long) (rho*k2+1.0);
    rr=(double) r; ss=(double) s;
    gg=(double) g; hh=(double) h; tt=(double) t;
    m1=floor(lam/(1.0-mu1));
    m2=floor(lam/(1.0-mu2));
    Z0=0.5*((m1*m1+m1)*(1.0-mu1)+(m2*m2+m2)*(1.0-mu2)-
        -hh*hh+hh-(1.0-mu1-mu2)*(gg*gg+gg));
    Z1=hh+gg-m1-m2-1.0;
    if (Z1<0.0) H=Z0+lam2*Z1;
    else H=Z0+lam1*Z1;
    reta=xi*pow(gg,1.5);
    E1=0.001*k2;
    E2=0.5*tt*(tt-1.0)+hh*tt*exp(-ss/(hh*tt))+ss*ss/(2.0*tt*reta);
    E3=log(Y*lam1*lam1)/(1.0*Y*lam1*lam1*lam1*lam1);
    *ex=(-E3+(1.0/(2.0*rr*ss))*(H-mu1*E1-mu2*E2))*lam1*lam1;
    logC1=th*k2*kk*logk;
    logC2=ss*ss/tt+10.5*xi*xi*tt*gg*gg*log(gg)*log(gg)/D;
    logC2 -=(ss*log(0.1*reta)*((reta+hh)*pow(1.0-1.0/hh,ss/tt)-h));
    logC3=1.0417*reta*log(10.4167*reta);
    logC=logC3/rr+(4.65*lam2*log(lam2)+logC1+logC2)/(2.0*rr*ss);
    *c=exp(logC)+1.0/kk;
    if (pr==1){
        printf("%8.4f-%8.4f %4ld",lam1,lam2,k);
        if(g>0) printf("%3ld %2ld %2ld %2ld %9.4f %7.4f\n",
                       s,g-g0,h1-h,t,1.0/(*ex)+0.00005,*c+0.00005);
        else printf("\n");
    }
}
int main()
{
    double E,lam8,lam9,r[9],tmp,maxex,con,maxcon
    double bestth,bestcon,bp[5000];
    long i,j,i0,w,n,m,maxm,bestg,besth,bests,s0,s1;
    mu1=0.1905; mu2=0.1603;
    goal=132.94357;
    while (1) {
        printf("enter Y:"); scanf("%lf", &Y);
        D=0.1019*Y;
        printf("enter xi: "); scanf("%lf", &xi);
        printf("enter sigma :"); scanf("%lf", &sigma);
        if (sigma<0.0) flag=1; else flag=0;
        printf("enter lambda range: "); 
        scanf("%lf %lf", &lam8, &lam9);
        if ((lam9<lam8)||(lam8<=80.0)||(lam9>=300.0)) continue;
        printf(" approx. \n");
        printf("lambda range    k   s   a   b   t   exp const\n");
        printf("-------------- --- --- --- --- --- -----------\n");
        bp[1]=lam8; bp[2]=lam9; j=3;
        i0=(long) (lam9/(1.0-mu1-mu2))+10;
        for (i=1; i<=i0;i++){
            w=(double) i;
            r[1]=w*(1.0-mu1);
            r[2]=w*(1.0-mu2);
            r[3]=(w-0.000003)*(1.0-mu1-mu2);
            for (m=1;m<=3;m++) if ((r[m]<lam9) && (r[m]>lam8)) 
            bp[j++]=r[m];
        }
        n=j-1;
        for (i=1;i<=n-1;i++) for(j=i+1;j<=n;j++)
        if (bp[j]<bp[i]) {tmp=bp[i]; bp[i]=bp[j]; bp[j]=tmp;}
        maxex=0.0;
        maxcon=0.0;
        for (j=1;j<=n-1;j++){
            lam=0.5*(bp[j]+bp[j+1]);
            lam1=bp[j]; lam2=bp[j+1];
            g0=(long) (lam/(1.0-mu1)+1.0); g1=g0+1;
            h1=(long) (lam/(1.0-mu2)); h0=h1-1;
            bestg=-1; besth=-1; bestth=1.0e20; bestcon=1.0e40;
            for (g=g0;g<=g1;g++) for (h=h0;h<=h1;h++){
                t=g-h+1;
                if ((g>=100) && ((double) g<=1.254*lam1)){
                    if (flag==0) {
                        s0=(long) (sigma*h*t+1.0); s1=s0;
                    }
                    else {
                        s0=h*(t-1)/4;
                        s1=h*t/2;
                    }
                    for(s=s0; s<=s1; s++) {
                        calc(&E,&con,0);
                        if((E>0.0)&&(1.0/E<goal)&&(con<bestcon)){
                           bestth=1.0/E; bestg=g; besth=h; bests=s;  
                        }
                    }
                }
            }
        
        g=bestg; h=besth; t=g-h+1;
        s=bests;
        calc(&E,&con,1);
        if (1.0/E>maxex) maxex=1.0/E;
        if (con>maxcon) maxcon=con;
    }
    printf("max. ex: %10.6f max. const.: %10.6f\n", maxex,maxcon);
  }
}
\end{lstlisting}
\section*{Acknowledgements}
I would like to thank my supervisor Timothy S. Trudgian for his support and helpful suggestions
throughout the writing of this article.

\clearpage
\printbibliography

\end{document}